\documentclass[12pt, twoside]{article}
\usepackage{amsmath,amsthm,amssymb}
\usepackage{times}
\usepackage{enumerate}

\def\titlerunning#1{\gdef\titrun{#1}}
\makeatletter
\def\author#1{\gdef\autrun{\def\and{\unskip, }#1}\gdef\@author{#1}}
\def\address#1{{\def\and{\\\hspace*{18pt}}\renewcommand{\thefootnote}{}%
\footnote {#1}}%
\markboth{\autrun}{\titrun}}
\makeatother
\def\email#1{e-mail: #1}
\def\subjclass#1{{\renewcommand{\thefootnote}{}%
\footnote{\emph{Mathematics Subject Classification (2010):} #1}}}
\def\keywords#1{\par\medskip
\noindent\textbf{Keywords.} #1}

\numberwithin{equation}{section}

\usepackage{amsfonts}
\usepackage{latexsym}
\usepackage{graphicx}
\usepackage{layout}
\usepackage{mathrsfs}
\usepackage{color}

\newtheorem{theorem}{Theorem}[section]
\newtheorem{proposition}[theorem]{Proposition}
\newtheorem{lemma}[theorem]{Lemma}
\newtheorem{corollary}[theorem]{Corollary}
{\theoremstyle{definition}\newtheorem{remark}[theorem]{Remark}}%JEMS

\newcommand{\sezione}[1]{\section{#1}\setcounter{equation}{0}}
\newcommand{\sottosezione}[1]{\subsection{#1}}

\newcommand{\nor}{\Arrowvert}
\newcommand{\rosso}{\color{black}}

\def\R{{\rm I\mskip -3.5mu R}}
\def\N{{\rm I\mskip -3.5mu N}}

\def\Om{\Omega}
\def\C2al{C^{2,\a}_{loc}}%\left( \bar \Om \setminus \{k_1,\dots,k_m\}\right)}

\def\intO{\int_{\Om}}
\def\ra{{\longrightarrow\,}}
\def\na{\nabla}
\def\d{\delta}
\def\l{{\lambda}}
\def\a{{\alpha}}

\def\S{{\mathcal{S}}}
\def\un{u_n}

\def\tv{\tilde v}
\def\tu{\tilde u}
\def\tx{\tilde x}

\def\c{\mathfrak{c}}

\def\C{\mathfrak{C}}
\def\ln{\l_n}

\def\vn{v_n}
\def\mn{\mu_n}

\def\de{\partial}

\usepackage{bm}
\def\ka{\kappa}
\def\a{\mathfrak{a}}
\def\b{\mathfrak{b}}
\def\A{\mathfrak{A}}
\newcommand{\vep}{\varepsilon}

\newcommand{\lap}{\Delta}
\newcommand{\del}{\partial}
\newcommand{\tend}{\longrightarrow}

\newcommand{\bs}[1]{\boldsymbol{#1}}

\newcommand{\Le}[2]{L^{#1}\left(#2\right)}

\newcommand{\Co}[2]{C^{#1}\left(#2\right)}

\newcommand{\clo}[1]{\overline{#1}}
\newcommand{\norm}[2]{\|#1\|_{#2}}
\newcommand{\ball}[2]{B_{#1}\left(#2\right)}
\newcommand{\til}[1]{\widetilde{#1}}

\begin{document}

%%%%% To ease editing, add:

\baselineskip=17pt

%%%%%%%%%%%%%%%%

%% In the running head, give an abbreviation of the title. 
\titlerunning{Morse indices of multiple blow-up solutions to the Gel'fand problem}

\title{Morse indices of multiple blow-up solutions \\
 to the two-dimensional Gel'fand problem}

\author{
Francesca Gladiali
\and 
Massimo Grossi
\and 
Hiroshi Ohtsuka
\and 
Takashi Suzuki}

\date{October 3, 2012}

\maketitle

\address{
F. Gladiali: Universit\`a degli Studi di
Sassari,via Piandanna 4 -07100 Sassari, Italy; \\ \email{fgladiali@uniss.it}
\and
M. Grossi: Dipartimento di Matematica, Universit\`a di Roma
``La Sapienza", P.le A. Moro 2 - 00185 Roma, Italy; \email{grossi@mat.uniroma1.it}
\and 
H. Ohtsuka (Corresponding author): Department of Applied Physics,
Faculty of Engineering,
University of Miyazaki,
Gakuen Kibanadai Nishi 1-1,
Miyazakishi, 889-2192, Japan; \email{ohtsuka@cc.miyazaki-u.ac.jp}
\and
T. Suzuki: Division of Mathematical Science,
Department of System Innovation,
Graduate School of Engineering Science,
Osaka University,
Machikaneyamacho 1-3,
Toyonakashi 560-8531, Japan; \\ \email{suzuki@sigmath.es.osaka-u.ac.jp}
}


%%%%%%%%

\begin{abstract}
Blow-up solutions to the two-dimensional Gel'fand problem are
studied. It is known that the location of  the blow-up points of
these solutions is related to a Hamiltonian function involving the Green function of the domain. We show that this implies an equivalence 
between the Morse indices of the solutions and the associated criticalpoints of the Hamiltonian. 

%% Keywords are optional
\end{abstract}

\sezione{Introduction}
\label{s1}
The purpose of the present paper is to study the Gel'fand problem
\begin{equation}
\label{1}
            \left\{\begin{array}{lc}
                        -\Delta u=\l e^{u}  &
            \mbox{  in }\Om\\
              u=0 & \mbox{ on }\partial \Om,
                      \end{array}
                \right.
\end{equation}
where $\Om\subset \R^2$ is a bounded domain with smooth boundary
$\de \Om$ and $\l>0$ is a parameter.  This problem is associated
with several phenomena in differential geometry, turbulence
theory, and gauge field theory (see \cite{suzuki08} and the
references therein.)

Let ${\cal C}=\{
(\lambda,u)\in\R^+\times C(\Omega) \mid (\ref{1})\hbox{ is
satisfied}\}$ be the solutions set. The first observation is that
${\cal C}=\emptyset$ for $\lambda$ large enough.  The next one is
 that ${\cal C}\cap \{ \lambda\geq \varepsilon\}$ is
compact in ${\bf R}\times C(\overline{\Omega})$ for any
$\varepsilon>0$ and then there are at least two solutions for each $0<\lambda \ll 1$ (see {\cite{cr75}} for this fact and also \cite{CL01} for more detailed construction of the solutions). 
%A  result (by \cite{CL01}) concerning
%the existence of the solution  states that there exist at least
%two solutions for each $0<\lambda \ll 1$.  

The structure of
${\cal C}$, however, is much richer according to the topological
and geometrical properties of the domain $\Omega$ (see
\cite{EGP05} ,\cite{DKM05} and \cite{s92}), which provides
significant effects to the above mentioned geometric and physical
theories. Critical phenomena in fact occur to the solution
$u=u(x)$ as $\lambda\downarrow 0$. This profile is described by
\cite{NS90} as a quantized blow-up mechanism.

Let $\{\ln\}_{n\in\N}$ be a sequence of positive values such that $\ln\to 0$ as $n\to \infty$ and let $\un=\un(x)$ be a sequence of solutions of \eqref{1} for $\l=\ln$. In \cite{NS90}, the authors proved the total mass quantization, that is,
\begin{equation}\label{2.1}
\ln \intO e^{\un}\, dx\rightarrow 8\pi m
\end{equation}
for some $m=0,1,2, \cdots, +\infty$ along a sub-sequence.\\
If $m=0$ the pair $({\lambda_n},u_{\lambda_n} )\in{\cal C}$
converges to
$(0,0)$ as ${\lambda_n}\rightarrow0$.\\
If $m=+\infty$ there arises the entire blow-up of the solution
$u_n$, in the sense that
$\inf_Ku_n\rightarrow+\infty$ for any $K\Subset \Omega$. \\
If $0<m<\infty$ the solutions $\{ u_{n}\}$ blow-up at $m$-points.
Thus there is a set ${\cal S}=\{ \kappa_1, \cdots,
\kappa_m\}\subset \Omega$ of $m$ distinct points such
that  $\nor \un \nor_{L^{\infty}(\omega)}=O(1)$ for any
$\omega\Subset \overline{ \Omega}\setminus \S$,
$$\un{|_{\S}}\rightarrow +\infty \quad \hbox{ as }n\to \infty, $$
and
\begin{equation}\label{2.2}
\un \rightarrow \sum_{j=1}^m 8\pi \,G(\cdot, \ka_j)\quad \hbox{ in }C^{2}_{loc}(\overline{ \Omega} \setminus \S).
\end{equation}
Here and henceforth, $G(x,y)$ denotes the Green function of
$-\Delta$ in $\Om$ with Dirichlet boundary condition.  The Robin
function $R(x)=K(x,x)$ is now defined by the regular part of
$G(x,y)$ denoted by $K=K(x,y)$, i.e.,
\begin{equation}
\label{2.3}
G(x,y)=\frac 1{2\pi}\log{|x-y|^{-1}}+K(x,y).
\end{equation}
Then the blow-up points satisfy,
\begin{equation}
\na H^m (\ka_1,\dots ,\ka_m)=0
\label{conditionS}
\end{equation}
where
$$H^m(x_1,\dots, x_m)=\frac 12 \sum_{j=1}^m R(x_j)+\frac 12 \sum_{\substack   {1\leq j,h\leq m\\ j\neq h}}G(x_j,x_h).$$
See also \cite{MW01} for relating facts.

%From the viewpoint of the elliptic theory, (\ref{2.2}) with (\ref{conditionS}) classify the singular limits. This
%classification is sharp in the sense that
If the critical point
$(\kappa_1, \dots, \kappa_m)$ of $H^m$ is non-degenerate, then it
generates a family of solutions $\{ u_\lambda\}_\lambda$ to
(\ref{1}) satisfying (\ref{2.1}) as $\lambda\downarrow 0$ (see
\cite{bp98}). {\rosso{Moreover}} the non-degeneracy of
$(\kappa_1, \dots, \kappa_m)$ implies that of $u_\lambda$ for
$0<\lambda \ll 1$.  This was
proven first for $m=1$ by \cite{gg04} and then by \cite{GOS11} for
the general case.  
The purpose of this paper is to know more about this correspondence between the solution $u_\lambda$ and the associated critical point $(\kappa_1, \dots, \kappa_m)$ of $H^m$ up to the Morse indices of the both (see Remark \ref{rem_turbulence} and Remark \ref{rem_bifurcation} for a more detailed motivation).

To state our results, we take a sequence of solutions $\{ \un\} $
to (\ref{1}) for $\lambda=\lambda_n$ satisfying
$\lambda_n\downarrow 0$ and \eqref{2.1}. Since $m=1$ was studied
in \cite{GG09} we shall assume $m\geq 2$ in the sequel. Then we
consider the eigenvalue problem
\begin{equation}\label{autov-n}
\left\{\begin{array}{lc}
                        -\Delta \vn^k= \mn^k\ln e^{\un}\vn^k  & \mbox{ in }\Om\\
\nor \vn^k\nor_{\infty}=\max_{\overline{ \Om}}\vn^k=1 &\\
                  \vn^k=0 & \mbox{ on }\partial \Om
                      \end{array}
                \right.
\end{equation}
%\begin{equation}
%\label{3}
%\left\{\begin{array}{lc}
%                        -\Delta v= \mu\ln e^{\un}v  & \mbox{ in }\Om\\
%                  v=0 & \mbox{ on }\partial \Om
%                      \end{array}
%                \right.
%\end{equation}
which admits a sequence of eigenvalues $\mn^1<\mn^2\leq\mn^3\leq  \dots$, where $\vn^k$ is the $k$-th eigenfunction of \eqref{autov-n} corresponding to the eigenvalue $\mn^k$.  %We may assume that
The Morse and augmented Morse index of $u_n$, denoted by
$\mathrm{ind}_M(u_n)$ and $\mathrm{ind}_M^\ast(u_n)$,
respectively, are defined by
$$
\mathrm{ind}_M(u_n)=\#\{k\in\N\,;\,\mn^k<1\},\quad \mathrm{ind}_M^\ast(u_n)=\#\{k\in\N\,;\,\mn^k\leq 1\}.
$$
Given a $C^2$-function $f$ of 2m-variables
\[ (x_1,\cdots, x_m)=(x_{1,1},x_{1,2},\cdots, x_{m,1},x_{m,2})\in \R^{2m}, \]
and its critical point $(\kappa_1, \cdots, \kappa_m)\in \R^{2m}$,
the Morse and augmented Morse index of $f$ at
$(\ka_1,\cdots,\ka_m)$ are denoted by $\mathrm{ind}_M
f(\ka_1,\cdots,\ka_m)$ and $\mathrm{ind}_M^\ast
f(\ka_1,\cdots,\ka_m)$, that is,
\begin{align*}
&\mathrm{ind}_M f(\ka_1,\cdots,\ka_m)=\#\{k\in\N\,;\,\Lambda^k<0\},\\
&\mathrm{ind}_M^\ast f(\ka_1,\cdots,\ka_m)=\#\{k\in\N\,;\,\Lambda^k\leq 0\},
\end{align*}
where $\Lambda^1\leq\Lambda^2\leq \cdots\leq \Lambda^{2m}$ are the
eigenvalues of the Hessian matrix
$\mathrm{Hess}f=\left(\frac{\de^2f}{\de x_{i,\alpha}\de
x_{j,\beta}}\right)$ at $(\ka_1,\cdots,\ka_m)$, with $i,j=1,..,m$
and $\alpha,\beta=1,2$.

\bigskip

Under these notations we can state our main result.
\begin{theorem}\label{t1}
Suppose that $\{u_n\}$ is a sequence of  solutions to \eqref{1} which blows-up at $\ka_1,\cdots,\ka_m\in\Omega$. Then its Morse index $\mathrm{ind}_M(u_n)$ and the augmented Morse index $\mathrm{ind}_M^\ast(u_n)$ satisfy the following estimates for $n$ large,
\begin{align}
&m+\mathrm{ind}_M\{-H^m(\ka_1,\cdots,\ka_m)\}\leq \mathrm{ind}_M(u_n), \label{eqn:t1}\\
&\mathrm{ind}^\ast_M(u_n)\leq m+\mathrm{ind}^\ast_M \{-H^m(\ka_1,\cdots,\ka_m)\}.
 \label{eqn:t2}
\end{align}
If $(\ka_1,\cdots,\ka_m)$ is a non-degenerate critical point of
$H^m$, it holds that $\mathrm{ind}_M
H^m(\ka_1,\cdots,\ka_m)=\mathrm{ind}^\ast_M
H^m(\ka_1,\cdots,\ka_m)$, and hence
$$
\mathrm{ind}_M(u_n)%=\mathrm{ind}^\ast_M(u_n)
=m+\mathrm{ind}_M\{-H^m(\ka_1,\cdots,\ka_m)\}.
$$
\end{theorem}
%\begin{remark}
%We mean that $u_n$ is non-degenerate if the linearized eigenvalue
%problem \eqref{autov-n} around $u_n$ does not admit any eigenvalue equal
%to $1$. This non-degeneracy property was proven in \cite{GOS11}.
%\end{remark}
%\begin{remark}
From the proof of Theorem \ref{t1} described above, we always have
$m\leq \mathrm{ind}_M(u_n)\leq \mathrm{ind}^\ast_M(u_n)\leq 3m$. A
direct proof of the first inequality, $m\leq \mathrm{ind}_M(u_n)$,
is given in \cite{ft}.\\[.5cm]
%\end{remark}
The previous result is a consequence of a delicate asymptotic
expansion of the first $3m+1$ eigenvalues. This result, contained
in the next theorem, %according to us
is interesting in itself.
\begin{theorem}\label{t1a}
We have that, for $\lambda_n\rightarrow0$,
\begin{equation}\label{10}
 \mn^k=-\frac 12 \frac 1{\log \ln}+o\left(\frac 1{\log \ln}\right)\quad\text{for $1\leq k\leq
 m$},
\end{equation}
\begin{equation}\label{10_1}
 \mn^{k}=1-48\pi\eta^{2m-(k-m)+1}\ln+o\left(\ln\right)\quad\text{for $m+1\leq k\leq 3m$},
\end{equation}
and
\begin{equation}\label{10_2}
\mn^k>1\quad\text{for $k\geq 3m+1$}
\end{equation}
where $\eta^k$\,$(k=1,\cdots,2m)$ is the $k$-th eigenvalue of the
matrix $D(\mathrm{Hess}H^m)D$ at $(\ka_1,\cdots,\ka_m)$. Here $D=(D_{ij})$ is the diagonal matrix $\mathrm{diag}[d_1,d_1,d_2,d_2,\cdots,d_m,d_m]$
 (see \eqref{2.6b} for
the definition of the constants $d_j$). % and $\mathrm{Hess}(DH^mD)$ denotes the Hessian matrix of $DH^mD$.
\end{theorem}

Theorem \ref{t1a} %follows by the classical mini-max principle and the proof
involves delicate computations. One of the crucial point is to
localize $u_n$ and its partial derivatives around the blowup
points $\kappa_1, \cdots, \kappa_m$. Actually, we will use them as
test functions to estimate the first $3m+1$ eigenvalues.

\begin{remark}
An analogous result to Theorem \ref{t1} has been proved in \cite{BYR95} for positive solutions of the problem
\begin{equation}\label{blr}
            \left\{\begin{array}{lc}
                        -\Delta u=u^{\frac{N+2}{N-2}-\epsilon}  &
            \mbox{  in }\Om\\
              u=0 & \mbox{ on }\partial \Om,
                      \end{array}
                \right.
\end{equation}
where $\Omega\subset\R^N$ is a smooth bounded domain, $N\ge4$ and
$\epsilon$ is small enough. However, the approach used in
\cite{BYR95} is quite different from ours and it does not provide
the estimates of Theorem \ref{t1a}. Similar estimates to
\eqref{10}-\eqref{10_2} for the problem \eqref{blr} was obtained
in \cite{GP05}.
\end{remark}

\begin{remark}
\label{rem_turbulence}
We note that $H^m=H^m(x_1, \dots, x_m)$ 
appears as %casts 
the Hamiltonian in the point vortex theory of Onsager \cite{ons49}.  In this theory
the vortex system
\begin{equation}
\frac{dx_i}{dt}=\nabla^\perp H^m(x_1, \dots, x_m), \quad i=1,\dots, m
 \label{eqn:vs}
\end{equation}
is used to describe the motion of point vortices $\omega(dx,t)=
\sum_{i=1}^N\delta_{x_i(t)}(dx)$ of perfect fluid in {\rosso{a two dimensional}} space.
Then the Gel'fand problem (\ref{1}) arises as a
high-energy limit {\rosso{as}} $N\rightarrow\infty$ in (\ref{eqn:vs}) under
the factorization property, sometimes called the propagation of
chaos \cite{jm73, pl76, es93}. We thus regard (\ref{eqn:vs}) as a
Hamilton system to take the canonical measure by a thermodynamical
relation (see \cite{clmp92, clmp95, kie93, s92} for more rigorous
approach and related mathematical results).

%The above described quantized blowup mechanism to (\ref{1}) reads,
%first, that the particle distribution clusters to several spots at
%the critical level $\lambda=0$. Next, this clustered states is a
%sum of copies of a delta function with the quantized mass $8\pi$,
%and hence each of them is regarded as a particle.  In this
%context, equality (\ref{conditionS}) indicates that the
%Hamiltonian controls these critical states beyond the material
%hierarchy.  One of our motivations is to examine this property up
%to the dynamical level, more precisely, the exact relation of the
%Morse index between the $m$-point blowup solution $u=u(x)$ to
%(\ref{1}) and its associated critical point $(\kappa_1, \dots,
%\kappa_m)$ of $H^m=H^m(x_1,\dots, x_m)$.
\end{remark}

\begin{remark}
\label{rem_bifurcation}
At this stage it may be worth mentioning of \cite{gt10}, where the
authors proved that if $\Omega$ is convex only $m=1$ is admitted
and $H^1$ has only one critical point.
However, bifurcation of critical points of $H^m$ may %,
occur if we perturb the domain into a non-convex one,
%(see \cite{EGP05} and \cite{DKM05}),
which implies the existence of the singular limits with $m>1$ 
%\[ u^\ast=\sum_{j=1}^m8\pi G(\cdot,\kappa_j) \]
%to (\ref{1})
as $\lambda\rightarrow0$ (see, e.g., \cite{ms97}, \cite{ccl03}, \cite{EGP05}, and \cite{DKM05}). %\\
In the generic case, these critical points of $H^m$ after bifurcation are
non-degenerate %in a neighborhood of the bifurcation points
and hence the associated singular limits generate non-degenerate classical solutions to (\ref{1}) for $\lambda$ small %}} \\
as we mentioned before. 
Therefore it may be natural to ask whether the change of Morse
%index
indices of the solutions $\{ u_\lambda\}$ %, at a bifurcation point,
follows from %that of $(\kappa_1, \dots, \kappa_m)$ or not.
the bifurcation of the critical point of $H^m$ or not.
%If the answer is affirmative, there must be a correspondence
%between the Morse index of $u_n$ as a solution to
%(\ref{1}) and the one of $(\kappa_1, \cdots, \kappa_m)$ as a
%critical point of $H^m$. 
%The purpose of the present paper is to
%show that this delicate property occurs.
The conclusion of the present paper supports this delicate property.

\end{remark}

\medskip

%Note that the sign of the $l$-th eigenvalue of
%$D\mathrm{Hess}(H^m)D$ coincides with that of
%$\mathrm{Hess}\,H^m(\ka_1,\cdots,\ka_m)$ (see Lemma
%\ref{sign_eigenvalue}). Consequently it follows that
%\begin{align}
%&\mathrm{ind}_M\{-H^m(\ka_1,\cdots,\ka_m)\}=\#\{k\in\N\,;\, \eta^k>0\}, \label{eqn:13} \\
%&\mathrm{ind}^\ast_M\{-H^m(\ka_1,\cdots,\ka_m)\}=\#\{k\in\N\,;\, \eta^k\geq 0\}. \label{eqn:14}
%\end{align}
 %for $n\gg 1$. Combining the asymptotic behaviors (\ref{10}),
%(\ref{10_1}), and (\ref{10_2}) with the relation
%(\ref{eqn:13})-(\ref{eqn:14}), we obtain
%(\ref{eqn:t1})-(\ref{eqn:t2}).

The quantized blowup mechanism (\ref{2.1}) induces the invariance
of the total degree of the set of solutions to the mean field
equation in dis-quantized intervals of the parameter, and the
degree is related to the genus of $\Omega$ (see
\cite{cl03}).  In this paper we are concentrated on the Gel'fand
problem (\ref{1}), although similar correspondences between the
Morse index of the solution and the Hamiltonian are suspected for
the mean field equation, too.  Our analysis here uses %some
 Y.Y. Li's estimate \cite{YYL99} (see \cite{L07} for an alternative
proof).%, while sharper estimates due to \cite{cl02} are not used.

This paper is organized as follows. Section \ref{s3} contains some
preliminaries and some estimates on the the eigenvalues $\mn$ when
$\mn\ra 0$ and $\mn\ra 1$ which will use in the next sections. In Section \ref{s4} we show the main estimates on the eigenvalues.
In Section \ref{se-t} we prove Theorems \ref{t1} and \ref{t1a}.
 %First, we show $\mn^1\ra 0$ and then $\mn^k\ra 0$, $2\leq k\leq m$, by an induction.
 %Next we show $\mn^{m+1}\ra 1$ and then $\mn^k\ra 1$, $m+2\leq k\leq 3m$, by an induction.
 %Meanwhile sharp asymptotic estimates of these eigenvalues described above are proven.
 In the appendix we show several elementary facts used in the paper.%, such as \eqref{DHD}
%%%%%%%%%%%%%%%%%%%%%%%%%%%%%%%%%SEZIONE
\sezione{Preliminaries and asymptotic estimates}\label{s3}
%\sezione{A first estimate on eigenvalues and eigenfunctions.}\label{s3}
%\section{Rough behavior of eigenvalues and eigenfunctions.}\label{s3}
%%%%%%%%%%%%%%%%%%%%%%%%%%%%%%%%%%%%%%%%%%%%%%%%%%%%%%%%%%%%%%
\sottosezione{The general case $\mn^k\geq 0$}
In this section we show several {\rosso {properties on }}
%general asymptotic profiles arising in
eigenvalues {\rosso {$\{\mn^k\}$ and eigenfunctions $\{\vn^k\}$}} of \eqref{autov-n} for any $k\geq 0$.

Take $0<R\ll 1$ satisfying $B_{2R}(\ka_i)\Subset \Omega$ for $i=1,\dots,m$ and $B_R(\ka_i)\cap B_R(\ka_j)=\emptyset$ if $i\neq j$. For each $\ka_j\in \S$, $j=1,\dots,m$, there exists a sequence $\{x_{j,n}\}\in B_R(\ka_j)$ such that
\begin{itemize}
\item[$ i)$] $\un(x_{j,n})=\sup_{B_R(x_{j,n})}\un(x)\rightarrow +\infty$,
\item[$ii)$] $x_{j,n}\to \ka_j$ as $n\to +\infty$.
\end{itemize}
Then we rescale $\un$ around $x_{j,n}$ as
\begin{equation}\label{2.4}
\tu_{j,n}(\tx):= \un\left( \d_{j,n} \tx +x_{j,n}\right)- \un(x_{j,n})\quad \hbox{ in }B_{\frac{R}{\d_{j,n}}}(0)
\end{equation}
where the scaling parameter $\d_{j,n}$ is determined by
\begin{equation}\label{2.5}
\ln e^{\un(x_{j,n})}\d_{j,n}^2=1.
\end{equation}

By \cite[Corollary 4.3]{GOS11} there exists a constant $d_j>0$ such that
\begin{equation}\label{2.6b}
\d_{j,n}=d_j \ln^{\frac 12}+o\left( \ln^{\frac 12}\right)
\end{equation}
as $n\to \infty$ for a sub-sequence, and in particular, $\delta_{j,n}\tend 0$.  Then relations \eqref{2.5} and \eqref{2.6b} in turn give
\begin{equation}\label{2.6c}
\un(x_{j,n})=-2 \log \ln +O(1)
\end{equation}
as $n\to \infty$ for any $j=1,\dots,m$.
%The modified Hesse matrix used in the previous section is defied by the above $d_j$.
%First, let $D=(D_{ij})$ be the diagonal matrix $\mathrm{diag}[d_1,d_1,d_2,d_2,\cdots,d_m,d_m]$.  Hence this $D$ takes the diagonal components{\rosso{ $D_{2j-1,2j-1}=D_{2j,2j}=d_j$ }}for $j=1,\cdots,m$. Then we put
%\begin{equation}\label{DHD}
%\widetilde{\mathrm{Hess}}\,H^m(\ka_1,\cdots,\ka_m)=D\mathrm{Hess}\,H^m(\ka_1,\cdots,\ka_m)D.
%\end{equation}

\begin{remark}
The above $d_j$ is determined by the blow-up set $\S$. % and hence \eqref{2.6b} arises without taking a subsequence.
Actually it {\rosso{can be proved}} that
\begin{equation}
d_j=\frac{1}{8}\exp\left\{ 4\pi R(\kappa_j)+4\pi\sum_{\substack   {1\leq i\leq m\\ i\neq j}}G(\ka_j,\ka_i)\right\}.
\label{value_dj}
\end{equation}
The proof of (\ref{value_dj}) requires a weak form of sharper estimates due to \cite{cl02}. % although this property is not necessary for the proof of Theorem \ref{t1}.% (see our forthcoming paper for details).
\end{remark}

The function $\tu_{j,n}$ in (\ref{2.4}) satisfies
\begin{equation}\nonumber
\left\{
\begin{array}{ll}
-\Delta \tu_{j,n}=e^{\tu_{j,n}} & \hbox{ in }B_{\frac{R}{\d_{j,n}}}(0)\\
\tu_{j,n}\leq \tu_{j,n}(0)=0& \hbox{ in }B_{\frac{R}{\d_{j,n}}}(0)
\end{array}
\right.
\end{equation}
and then a classification result (see \cite{CL91}) implies
\begin{equation}\label{2.6}
\tu_{j,n}(\tx)\rightarrow U(\tx)=\log \frac 1{\left( 1+\frac{|\tx|^2}8\right)^2} \quad \hbox{ in }C^{\infty}_{loc}(\R^2).
\end{equation}
Moreover (see \cite{YYL99}), it holds that
\begin{equation}\label{2.6a}
\big| \tu_{j,n}(\tx)- U(\tx)\big|\leq C, \quad \forall \tx \in B_{\frac{R}{\d_{j,n}}}(0)
\end{equation}
with a constant $C>0$.

As we did for $\un$, we rescale also the eigenfunctions $\vn$ around $x_{j,n}$ using $\d_{j,n}$ defined by \eqref{2.5}, that is,
\begin{equation}\label{2.8}
\tv_{j,n}(\tx):=\vn\left( \d_{j,n} \tx +x_{j,n}\right)\quad \hbox{ in }B_{\frac{R}{\d_{j,n}}}(0).
\end{equation}
Then it holds that
\begin{equation}\label{2.8a}
\left\{
\begin{array}{ll}
-\Delta \tv_{j,n}=\mn  e^{\tu_{j,n}}\tv_{j,n} & \hbox{ in }B_{\frac{R}{\d_{j,n}}}(0)\\
\nor \tv_{j,n} \nor_{L^{\infty}\big(B_{\frac{R}{\d_{j,n}}}(0)\big)}\leq 1.
\end{array}
\right.
\end{equation}
The following proposition follows 
%by from
from 
the standard elliptic theory.
\begin{proposition}\label{newp3.1}
Let $\{\mn\}$ be a sequence of eigenvalues of \eqref{autov-n} satisfying
$$
\mn\tend\mu_\infty\in\R
$$
as $n\ra \infty$. Then there exist $(V_1,\cdots,V_m)\in C^{2,\alpha}_{loc}(\R^2)^m$ and a sub-sequence satisfying
$$
\tv_{j,n}\ra V_j\quad\text{in $C^{2,\alpha}_{loc}\left(\R^2\right)$}
$$
and
\begin{equation}\label{prob-lim}
-\Delta V_j=\mu_\infty e^UV_j\quad\text{in $\R^2$},\quad \norm{V_j}{\Le{\infty}{\R^2}}\leq 1.
\end{equation}
\end{proposition}

\begin{remark}\label{newr3.2}
Since it may happen that $V_j\equiv 0$ {\rosso{ for any $j=1,\dots,m$}}, from \eqref{prob-lim} we cannot derive that $\mu_\infty$ is an eigenvalue of
\begin{equation}\label{new3.2}
-\lap V=\alpha e^UV\quad\text{in $\R^2$},\quad V\in\Le{\infty}{\R^2}.
\end{equation}
Later we shall prove that $V_j\not\equiv 0$ for some $j\in\{1,\cdots,m\}$, and consequently, $\mu_\infty$ is {\rosso{ actually}} an eigenvalue of \eqref{new3.2}.
\end{remark}

The structure of the eigenvalue problem (\ref{new3.2}) is described in \cite[Theorem 11.1]{GG09}.  All the eigenvalues are thus given as $\alpha_k=\frac{k(k+1)}{2}$ for $k=0,1,2,\cdots$, with the multiplicity $2k+1$.  To examine the Morse index of $u_n$, therefore, we need to study the cases $\mu_\infty=\alpha_0=0$ and $\mu_\infty=\alpha_1=1$.  %For these cases the eigenfunctions are given as follows:
\begin{remark}\label{remark-2.4} We have that
\begin{enumerate}
\item[$i)$] If $\mu_\infty=\alpha_0=0$ we have $V_j\equiv c_j$, where $c_j\in\R\backslash\{0\}$ is a constant.  Let $\c=(c_1,\cdots,c_m)$.
\item[$ii)$] If $\mu_\infty=\alpha_1=1$ there exists a vector $\bm{a}_j\in\R^2$ and a constant $b_j\in\R$ such that $(\bm{a}_j,b_j)\not=(\bm{0},0)$ and $V_j=\bm{a}_j\cdot\nabla U+b_j\clo{U}$, where $\clo{U}=\tx \cdot\nabla U+2$.  Let $\a=(\bm{a}_1,\cdots,\bm{a}_m)\in \R^{2m}$ and $\b=(b_1,\cdots,b_m)\in \R^m$.
\end{enumerate}
\end{remark}
\bigskip
In the next proposition we show the asymptotic profile of $v_n$ in $\clo{\Omega}\backslash\{\kappa_1,\cdots,\kappa_m\}$.
%%%%%%%%%%%%%%%%%%%DA QUI
\begin{proposition}\label{newp3.3}
For $0<R\ll 1$ it holds that
\begin{equation}\label{asin-v-n}
\frac{v_n(x)}{\mu_n}=\sum_{j=1}^m\left\{\gamma_{j,n}^0G(x,x_{j,n})+\bm{\gamma^1}_{j,n}\cdot\nabla_yG(x,x_{j,n})\right\}+o(\lambda_n^\frac{1}{2})
\end{equation}
uniformly in $\clo{\Omega}\backslash\cup_{j=1}^m\ball{R}{\kappa_j}$, where
\begin{align*}
\gamma_{j,n}^0=\int_{\ball{R}{x_{j,n}}}\ln e^{u_n}{v_n}\, dx
\end{align*}
and $\bm{\gamma^1}_{j,n}=(\gamma_{j,n}^{1,1},\gamma_{j,n}^{1,2})$ with
\begin{align*}
\gamma_{j,n}^{1,\alpha}=\int_{\ball{R}{x_{j,n}}}(x-x_{j,n})_\alpha\ln e^{u_n}{v_n}\, dx, \quad \alpha=1,2.
\end{align*}
\end{proposition}
\begin{proof}
We may assume $x_{j,n}\in\ball{\frac{R}{4}}{\kappa_j}$.  Using Green's representation formula, we have
\begin{align*}
&\frac{v_n(x)}{\mu_n}=\int_\Omega G(x,y)\ln e^{u_n}v_ndy\\
&=\sum_{j=1}^m\int_{\ball{\frac{R}{4}}{x_{j,n}}}G(x,y)\ln e^{u_n}v_ndy+\int_{\Omega\backslash\cup_{j=1}^m\ball{\frac{R}{4}}{x_{j,n}}}G(x,y)\ln e^{u_n}v_ndy.
\end{align*}
Here it holds that,
\begin{eqnarray*}
& & \left|\int_{\Omega\backslash\cup_{j=1}^m\ball{\frac{R}{4}}{x_{j,n}}}G(x,y)\ln e^{u_n}v_n\right| \\
& & \leq O(\ln)\int_{\Omega\backslash\cup_{j=1}^m\ball{\frac{R}{4}}{x_{j,n}}}|G(x,y)|dy=O(\ln).
\end{eqnarray*}
Taylor's theorem, on the other hand, implies
$$
G(x,y)=G(x,x_{j,n}) + (y-x_{j,n})\cdot\nabla_y G(x,x_{j,n}) + s(x,\eta,y-x_{j,n})
$$
for $x \in \clo{\Omega}\setminus \ball{R}{\kappa_j}
%(\Subset\clo{\Omega}\setminus \ball{\frac{R}{2}}{\kappa_j}\Subset
{\rosso{\Rightarrow x\in\clo{\Omega}\setminus \ball{\frac{R}{4}}{x_{j,n}}}}$ and $y \in \ball{\frac{R}{4}}{x_{j,n}}$, where
\begin{eqnarray*}
s(x,\eta,y-x_{j,n}) &=& \frac{1}{2} \sum_{1 \leq \alpha, \beta \leq 2} G_{y_\alpha y_\beta}(x, \eta) (y-x_{j,n})_\alpha(y - x_{j,n})_\beta, \\
\eta &=& \eta(j,n,y) \in \ball{\frac{R}{4}}{x_{j,n}}.
\end{eqnarray*}
Then we obtain
\begin{align*}
&\int_{\ball{\frac{R}{4}}{x_{j,n}}}G(x,y)\ln e^{u_n}v_ndy=G(x,x_{j,n})\int_{\ball{\frac{R}{4}}{x_{j,n}}}\ln e^{u_n}v_ndy\\
&\quad+\nabla_yG(x,x_{j,n})\cdot \int_{\ball{\frac{R}{4}}{x_{j,n}}}(y-x_{j,n})\ln e^{u_n}v_ndy\\
&\qquad+\frac{1}{2}\sum_{1 \leq \alpha, \beta \leq 2}\int_{\ball{\frac{R}{4}}{x_{j,n}}}(y-x_{j,n})_\alpha(y-x_{j,n})_\beta G_{y_\alpha y_\beta}(x, \eta) \ln e^{u_n}v_ndy.
\end{align*}

{\rosso{So}} we have
\begin{align*}
\int_{\ball{\frac{R}{4}}{x_{j,n}}}\ln e^{u_n}v_ndy&=\gamma_{j,n}^0-\int_{\ball{R}{\ka_j}\backslash\ball{\frac{R}{4}}{x_{j,n}}}\ln e^{u_n}v_ndy=\gamma_{j,n}^0+O(\ln)\\
&=\gamma_{j,n}^0+o\left(\ln^\frac{1}{2}\right)
\end{align*}
and also
$$
\int_{\ball{\frac{R}{4}}{x_{j,n}}}(y-x_{j,n})_\alpha\ln e^{u_n}v_ndy=\gamma_{j,n}^{1,\alpha}+O(\ln)=\gamma_{j,n}^{1,\alpha}+o\left(\ln^\frac{1}{2}\right).
$$
Finally, taking $\vep\in(0,1)$, we get
\begin{align*}
&\left|\int_{\ball{\frac{R}{4}}{x_{j,n}}}G_{y_\alpha y_\beta}(x, \eta)(y-x_{j,n})_\alpha(y-x_{j,n})_\beta  \ln e^{u_n}v_ndy\right|\\
&\leq CR^{\vep}\int_{\ball{\frac{R}{4}}{x_{j,n}}}|y-x_{j,n}|^{2-\vep}\ln e^{u_n}dy\\
&=CR^\vep\delta_{j,n}^{2-\vep}\int_{\ball{\frac{R}{4\delta_{j,n}}}{0}}|\til{y}|^{2-\vep} e^{\til{u}_{j,n}}d\til{y}=O\left(\delta_{j,n}^{2-\vep}\right)=o\left(\ln^\frac{1}{2}\right)
\end{align*}
by \eqref{2.6a}.  The proof is complete.
\end{proof}
\begin{remark}\label{newr3.4}
Using that
$$
\frac{\nabla v_n}{\mn}=\int_\Omega \nabla_x G(x,y)\ln e^{u_n}v_ndy,
$$
similarly we have
\begin{equation}\label{new3.3}
\frac{1}{\mu_n}\cdot\frac{\de v_n}{\de x_\alpha}=\sum_{j-1}^m\left\{\gamma_{j,n}^0G_{x_\alpha}(x,x_{j,n})+\bm{\gamma^1}_{j,n}\cdot\nabla_yG_{x_\alpha}(x,x_{j,n})\right\}+o\left(\lambda_n^\frac{1}{2}\right)
\end{equation}
for $\alpha=1,2$, uniformly in $\clo{\Omega}\backslash\cup_{j=1}^m\ball{R}{\kappa_j}$.
\end{remark}
\begin{remark}\label{newr3.5}
Since $\gamma_{j,n}^0=O(1)$ and $\bm{\gamma^1}_{j,n}=O(1)$, we have
{\rosso{
\begin{equation}\label{2.13}
v_n=O(\mn)\quad \text{ in }C^1\left( \clo \Omega\setminus \cup_{j=1}^m B_R(k_j)\right)
\end{equation}
for every $\mu_\infty\in \R$.}}
\end{remark}

% Then Proposition \ref{newp3.3} implies
%\begin{equation}\label{new3.4}
%v_n\tend 0\quad \text{in $\Co{1}{\clo{\Omega}\backslash\cup_{j=1}^m\ball{R}{\kappa_j}}$}
%\end{equation}
%if $\mn\ra\mu_\infty=0$. This property \eqref{new3.4} is valid even for $\mu_\infty\not=0$.  In fact, since
%\begin{align}
%|\gamma_{j,n}^{1,\alpha}|&\leq\int_{\ball{R}{x_{j,n}}}|y-x_{j,n}|\ln e^{u_n}dy
%\nonumber\\
%&=\delta_{j,n}\int_{\ball{\frac{R}{\delta_{j,n}}}{0}}|\til{y}|e^{\til{u}_{j,n}}d\til{y}=O(\delta_{j,n})=O\left(\ln^\frac{1}{2}\right)
%\label{new3.4.5}
%\end{align}
%we obtain \eqref{new3.4} with the next proposition.
%\end{remark}
%

\begin{proposition}\label{newp3.6}
If $\mu_\infty\not=0$ then it follows that
\begin{equation}\label{new3.5.5}
\gamma^0_{j,n}=O\left(\frac{1}{\log \ln}\right)
\end{equation}
{\rosso{and
\begin{equation}\label{gamma-1}
\bm{\gamma^1}_{j,n}=O\left(\ln^{\frac 12}\right).
\end{equation}}}
\end{proposition}
\begin{proof}
From equations \eqref{1} and \eqref{autov-n}, we have
\begin{align}
\int_{\de B_R(x_{j,n}) } &\Big( \frac{\de \un}{\de \nu} \frac{\vn}{\mn}-\un \frac{\de}{\de \nu}\frac{\vn}{\mn}\Big)\, dx=
\int_{B_R(x_{j,n}) }\left(\Delta \un \frac{\vn}{\mn}-\un \Delta \frac{\vn}{\mn}\right) \, dx\nonumber\\
=&-\ln \int_{B_R(x_{j,n}) }\!\!\!\!\!\!e^{\un}\frac{\vn}{\mn} \, dx+ \ln\int_{B_R(x_{j,n}) } \!\!\!\!\!\!e^{\un}\vn\un \, dx\nonumber\\
=&-\ln \int_{B_R(x_{j,n}) }\!\!\!\!\!\!e^{\un}\frac{\vn}{\mn} \, dx+ \un(x_{j,n})\ln\int_{B_R(x_{j,n}) }\!\!\!\!\!\! e^{\un}\vn \, dx\nonumber\\
&\qquad+ \ln \int_{B_R(x_{j,n}) }\!\!\!\!\!\! e^{\un}\vn \big\{ \un -\un(x_{j,n})\big\}\,dx\nonumber\\
=&\left(u_n(x_{j,n})-\frac{1}{\mn}\right)\gamma^0_{j,n}+ \ln \int_{B_R(x_{j,n}) }\!\!\!\!\!\!e^{\un}\vn \big\{ \un -\un(x_{j,n})\big\}\,dx.\label{stima}
\end{align}
Here it holds that
\begin{align}
\ln \int_{B_R(x_{j,n}) }\!\!\!e^{\un}\vn \big\{ \un -\un(x_{j,n})\big\}\,dx&=\int_{B_{\frac{R}{\delta_{j,n}}}(0) } e^{\til{u}_{j,n}}\til{v}_{j,n}\til{u}_{j,n}\,d\til{x}\nonumber
\\
&\ra\int_{\R^2}e^UV_jU=O(1),
 \label{2.14+1}
\end{align}
while \eqref{2.2} and \eqref{new3.3} imply
\begin{equation}\label{2.18a}
\int_{\de B_R(x_{j,n}) } \Big( \frac{\de \un}{\de \nu}\frac{\vn}{\mn}-\un \frac{\de}{\de \nu}\frac{\vn}{\mn}\Big)\, dx=O(1).
\end{equation}
Hence we obtain
$$
\gamma_{j,n}^0=\frac{O(1)}{u_n(x_{j,n})-\frac{1}{\mn}}=O\left(\frac{1}{\log \ln}\right)
$$
by \eqref{stima}, {\rosso{\eqref{2.6c} and}} under the assumption of $\mu_\infty\not=0$.
{\rosso{Moreover
\begin{align}
|\gamma_{j,n}^{1,\alpha}|&\leq\int_{\ball{R}{x_{j,n}}}|y-x_{j,n}|\ln e^{u_n}dy
\nonumber\\
&=\delta_{j,n}\int_{\ball{\frac{R}{\delta_{j,n}}}{0}}|\til{y}|e^{\til{u}_{j,n}}d\til{y}=O(\delta_{j,n})=O\left(\ln^\frac{1}{2}\right) .
\label{new3.4.5}
\end{align}}}
\end{proof}
\begin{corollary}
For every $\mu_\infty\in \R$ it holds that
\begin{equation}\label{new3.4}
v_n\tend 0\quad \text{in $\Co{1}{\clo{\Omega}\backslash\cup_{j=1}^m\ball{R}{\kappa_j}}$}.
\end{equation}
\end{corollary}
\begin{proof}
If $\mu_\infty=0$ \eqref{new3.4} follows by % Proposition \ref{newp3.3} and
Remark \ref{newr3.5}. Otherwise, if $\mu_\infty\neq 0$
then \eqref{new3.4} follows by Proposition \ref{newp3.3} and Proposition \ref{newp3.6}.
\end{proof}
\begin{remark}
For every $\mu_\infty\neq 0$ \eqref{2.18a} becomes
\begin{equation}\label{2.18b}
\int_{\de B_R(x_{j,n}) } \Big( \frac{\de \un}{\de \nu}\frac{\vn}{\mn}-\un \frac{\de}{\de \nu}\frac{\vn}{\mn}\Big)\, dx=o(1).
\end{equation}
This estimate will be useful in the sequel.
\end{remark}

%%%%%%%%%%%%%%%%%%%%%%%%SPOSTARE!!!!

%\begin{remark}
%Similarly to (\ref{new3.4.5}) we have
%\begin{equation}\label{new3.5}
%\frac{\gamma^1_{j,n}}{\delta_{j,n}}\ra \int_{\R^2}\til{y}e^U\left\{\bm{a}_j\cdot\nabla U+b_j\clo{U}\right\}=-8\pi\bm{a}_j
%\end{equation}
%in the case of $\mu_\infty=1$.  We have also the following proposition.
%\end{remark}
%

%\begin{remark}\label{newr3.7}
%If $\mu_\infty=0$ it holds that
%$$
%\gamma_{j,n}^0=\int_{\ball{R}{x_{j,n}}}\ln e^{u_n}{v_n}=\int_{\ball{\frac{R}{\delta_{j,n}}}{0}}e^{\til{u}_{j,n}}{\til{v}_{j,n}}\tend c_j\int e^U=8\pi c_j.
%$$
%In this case the argument used for the proof of Proposition \ref{newp3.6} gives a rough asymptotic behavior of $\mn$ (see Proposition \ref{newp3.10}).
%\end{remark}
%
%\bigskip
%
%The results obtained so far are the following: First, for every $\mu_\infty\in\R$, it holds that
%\begin{equation}\label{new3.9}
%v_n\tend 0\quad\text{in $\Co{1}{\clo{\Omega}\backslash\cup_{j=1}^m\ball{R}{\kappa_j}}$}.
%\end{equation}
%Next, we have
%\begin{equation}\label{new3.10}
%v_n=O(\mu_n)\quad\text{in $\Co{1}{\clo{\Omega}\backslash\cup_{j=1}^m\ball{R}{\kappa_j}}$}
%\end{equation}
%and
%\begin{align}
%v_n&=\sum_{j=1}^m\left\{\frac{8\pi b_j+o(1)}{\log\ln}G(x,x_{j,n})-8\pi d_j\lambda^\frac{1}{2}\bm{a}_j\cdot\nabla_y G(x,x_{j,n})\right\}+o\left(\lambda^\frac{1}{2}\right)\nonumber\\
%&\qquad\qquad\text{in $\Co{1}{\clo{\Omega}\backslash\cup_{j=1}^m\ball{R}{\kappa_j}}$}.
%\label{new3.11}
%\end{align}
%if $\mu_\infty=0$ and $\mu_\infty=1$, respectively.
%

\begin{proposition}\label{newp3.8}
There exists $j\in\{1,\cdots,m\}$ satisfying $V_j\not\equiv 0$.
\end{proposition}
%\begin{remark}\label{newr3.9}
%The asymptotics of eigenvalues and eigenfunctions are to be controlled precisely by Proposition \ref{newp3.8}.  It shows also that $\mu_\infty$ is an eigenvalue of the problem \eqref{new3.2} (see Remark \ref{newr3.2}).
%\end{remark}
%
\begin{proof}%[Proof of Proposition \ref{newp3.8}]
It is enough to show
\begin{equation}\label{new3.12}
v_n\tend 0\quad\text{uniformly in $\ball{R}{\kappa_j}$}
\end{equation}
if $V_j\equiv 0$.  In fact, if \eqref{new3.12} holds for all $j\in\{1,\cdots,m\}$ then we obtain
$$
v_n\tend 0\quad\text{uniformly in $\clo{\Omega}$}
$$
from \eqref{new3.4}, which contradicts the hypothesis $\norm{v_n}{\Le{\infty}{\Omega}}=1$.

Property \eqref{new3.12} is proven {\rosso{ using the same  argument}}  in \cite[(5.2)]{GOS11}.  If \eqref{new3.12} does not hold, we have
$$
\limsup_{n\to +\infty}\max_{x\in B_R(\ka_j)}|\vn (x)|=\limsup_{n\to +\infty}\max_{x\in B_{2R}(x_{j,n})}|\vn(x)|=M>0
$$
since $\ball{R}{\ka_j}\Subset\ball{2R}{x_{j,n}}$ for $n\gg 1$ and \eqref{new3.4}.  Let $\tilde{z}_{j,n}\in B_{\frac {2R}{\d_{j,n}}}(0)$ be the points such that
$$
\til{v}_{j,n}(\tilde{z}_{j,n})=\sup_{B_{\frac {2R}{\d_{j,n}}}(0)}|\til{v}_{j,n}(x)|.
$$
Up to a sub-sequences (denoted by the same symbol), it holds that
$$
|\tv_{j,n}(\tilde{z}_{j,n} )|=\max_{x\in B_{2R}(x_{j,n})}|\vn(x)|\to M, \quad |\tilde{z}_{j,n}|\to +\infty
$$
by $V_j\equiv 0$.

Here we take the Kelvin transform of $\tu_{j,n}$  and $\tv_{j,n}$, i.e.,
$$
\hat{u}_{j,n}(x)=\tu_{j,n}\Big(\frac{x}{|x|^2}\Big), \quad \hat{v}_{j,n}(x)=\tv_{j,n}\Big(\frac{x}{|x|^2}\Big)
$$
which satisfy
$$
-\Delta \hat{v}_{j,n}=\frac{\mn}{|\hat{x}|^4}e^{\hat{u}_{j,n}}\hat{v}_{j,n}\quad \hbox{ in }B_{\frac{\d_{j,n}}{2R}}(0)^c.
$$
Here we have
\begin{equation}\label{conv-z}
\hat{z}_{j,n}\in B_1(0)\setminus\overline{B_{\frac{\delta_{j,n}}{2R}}(0)}, \quad \hat{z}_{j,n}\to 0, \quad \hat{v}_{j,n}(\hat{z}_{j,n})=\tv_{j,n}(\tilde{z}_{j,n} )\to M.
\end{equation}
for $\hat{z}_{j,n}=\frac{\tilde{z}_{j,n}}{|\tilde{z}_{j,n}|^2}$. Let $w_{j,n}\in H^1_0(B_1(0))$ be such that
$$
\left\{ \begin{array}{ll}
-\Delta w_{j,n}=f_{j,n} & \hbox{ in }B_1(0)\\
w_{j,n}=0 & \hbox{ on }\de B_1(0)
\end{array}\right.
$$
where
$$
f_{j,n}:=\left\{ \begin{array}{ll}
\frac{\mn}{|\hat{x}|^4}e^{\hat{u}_{j,n}}\hat{v}_{j,n}
 & \hbox{ in }B_1(0)\setminus \overline{B_{\frac {\d_{j,n}}{2R}}(0)}\\
0 & \hbox{ in } B_{\frac {\d_{j,n}}{2R}}(0).
\end{array}\right.$$

We have, from \eqref{2.6a}
$$
0\leq  \frac{\mn}{|\hat{x}|^4}e^{\hat{u}_{j,n}}\leq C<\infty
$$
where $C$ is a constant independent on $n$.  We have, on the other hand, $ \hat{v}_{j,n}(\hat{x})=\tv_{j,n}( \til{x})\to 0$ for every $\hat{x}\in B_R(0)\setminus\{0\}$ by $V_j\equiv 0$.  Therefore, it holds that
$$
\nor f_{j,n}\nor_{L^p(B_1(0))}\to 0 \quad \hbox{ for each }p\in [1,+\infty)
$$
by the dominated convergence theorem which implies
$$
w_{j,n}\to 0\hbox{ uniformly in }B_1(0)
$$
from the elliptic theory.

We turn to the difference $\hat{v}_{j,n}-w_{j,n}$ which is harmonic in $B_1(0)\setminus\overline{B_{\frac {\d_{j,n}}{2R}}(0)}$.  Then the maximum principle guarantees
\begin{align*}
&\nor \hat{v}_{j,n}-w_{j,n}\nor_{L^{\infty}\big( B_1(0)\setminus\overline{B_{\frac {\d_{j,n}}{2R}}(0)}\big)}\\
&\leq \nor\hat{v}_{j,n}-w_{j,n}\nor_{L^{\infty}\big(\de B_1(0)\big)}+\nor\hat{v}_{j,n}-w_{j,n}\nor_{L^{\infty}\big(\de B_{\frac {\d_{j,n}}{2R}}(0)\big)}\\
&\leq \nor\hat{v}_{j,n}\nor_{L^{\infty}\big(\de B_1(0)\big)}+\nor\hat{v}_{j,n}\nor_{L^{\infty}\big(\de B_{\frac {\d_{j,n}}{2R}}(0)\big)}+\nor w_{j,n}\nor_{L^{\infty}\big(\de B_{\frac {\d_{j,n}}{2R}}(0)\big)}.
\end{align*}
Here, it follows from $V_j\equiv 0$ that
\begin{align*}
\nor\hat{v}_{j,n}\nor_{L^{\infty}\big(\de B_1(0)\big)}&=\nor \tv_{j,n}\nor_{L^{\infty}\big(\de B_1(0)\big)}=o(1)
\end{align*}
and from \eqref{new3.4}
\begin{align*}
\nor\hat{v}_{j,n}\nor_{L^{\infty}\big(\de B_{\frac {\d_{j,n}}{2R}}(0)\big)}&=\nor \tv_{j,n}\nor_{L^{\infty}\big(\de B_{\frac {2R}{\d_{j,n}}} (0)\big)}=\nor v_{n}\nor_{L^{\infty}\big(\de B_{2R} (x_{j,n})\big)}=o(1).
\end{align*}
Hence we obtain
\begin{align*}
\nor \hat{v}_{j,n}\nor_{L^{\infty}\big( B_1(0)\setminus\overline{B_{\frac {\d_{j,n}}{2R}}(0)}\big)}&\leq \nor w_{j,n}\nor_{L^{\infty}\big( B_1(0)\big)}+\nor \hat{v}_{j,n}-w_{j,n}\nor_ {L^{\infty}\big( B_1(0)\setminus\overline{B_{\frac {\d_{j,n}}{2R}}(0)}\big)}\\
&=o(1)
\end{align*}
which contradicts \eqref{conv-z}.
\end{proof}

{\rosso{\begin{corollary}\label{corollary-2.12-bis}
We have that $\mu_\infty$ is an eigenvalue of \eqref{prob-lim}.
\end{corollary}}}

\sottosezione{The case of $\mu_\infty=0$ and $\mu_\infty=1$}
In this section we consider the cases $\mu_\infty=0$ and $\mu_\infty=1$ and we improve the estimate of the previous section. First we start with a sharp estimate for $\mu_\infty=0$.

\begin{proposition}\label{newp3.10}
If $\mu_\infty=0$ it holds that
{\rosso{
\begin{equation}\label{*}
\gamma^0_{j,n}=8\pi c_j+o(1),
\end{equation}
and}}
$$
\mn=-\frac{1}{2\log\ln}+o\left(\frac{1}{\log\ln}\right).
$$
\end{proposition}
\begin{proof}
{\rosso{It holds that
$$\gamma^0_{j,n}=\int_{B_R(x_{j,n})}\l_n e^{u_n}v_n= \int_{B_{\frac R{\d_{j,n}}}(0) }e^{\tilde u_{j,n}}\tilde v_{j,n}\ra c_j \int_{\R^2}e^U=8\pi c_j.$$}}
We repeat the argument used for the proof of Proposition \ref{newp3.6}, using (\ref{stima}).  First, \[ \int_{\partial B_R(x_{j,n})}\left(\frac{\partial u_n}{\partial \nu}\frac{v_n}{\mu_n}-u_n\frac{\partial}{\partial\nu}\frac{v_n}{\mu_n}\right)=O(1) \]
holds by (\ref{2.2}) and (\ref{2.13}).  Next, the limit of (\ref{2.14+1}) is equal to $\int_{\R^2}e^UV_jU=c_j\int_{\R^2}e^UU$ in this case. Hence it follows that, {\rosso{using \eqref{*}}}
$$
O(1)=\left\{u_n(x_{j,n})-\frac{1}{\mn}\right\}\left\{8\pi c_j+o(1)\right\}+c_j\int_{\R^2}e^UU+o(1) .
$$
%from Remark \ref{newr3.7}.

Let $j\in\{1,\cdots,m\}$ be such that $c_j\not=0$, assured by Proposition \ref{newp3.8}. Then the above relation implies
$$
u_n(x_{j,n})-\frac{1}{\mn}=O(1)
$$
and consequently, the conclusion by \eqref{2.6c}.
\end{proof}

Now we consider the case of $\mu_\infty=1$. The proof of the asymptotic behavior of $\mu_n$ for $\mu_\infty=1$ will be performed in several steps.

\begin{proposition}
If $\mu_\infty=1$ then it holds that
\begin{equation}
\gamma^0_{j,n}=\frac{8\pi b_j+o(1)}{\log\lambda_n},
 \label{new3.6}
\end{equation}
\begin{equation}\label{new3.5}
\frac{\bm{\gamma^1}_{j,n}}{\delta_{j,n}}= -8\pi\bm{a}_j+o(1),
\end{equation}
and
\begin{align}
v_n&=\sum_{j=1}^m\left\{\frac{8\pi b_j+o(1)}{\log\ln}G(x,x_{j,n})-8\pi d_j\lambda^\frac{1}{2}\bm{a}_j\cdot\nabla_y G(x,x_{j,n})\right\}+o\left(\lambda^\frac{1}{2}\right)\nonumber\\
&\qquad\qquad\text{in $\Co{1}{\clo{\Omega}\backslash\cup_{j=1}^m\ball{R}{\kappa_j}}$}.
\label{new3.11}
\end{align}
 \label{prop:2.9}
\end{proposition}
\begin{proof}
By \eqref{stima} and  \eqref{2.18b} we get
$$
\gamma_{j,n}^0=\frac{\int_{\R^2}e^UV_jU+o(1)}{2\log \ln}.
$$
Since $\mu_\infty=1$, we obtain
$$
\int_{\R^2}e^UV_jU=\int_{\R^2}e^U\left(\bm{a}_j\cdot\nabla U+b_j\clo{U}\right)U=16\pi b_j
$$
by Lemma \ref{l6.2}, {\rosso{proving \eqref{new3.6}. \\
Similarly to (\ref{new3.4.5}) we have, from Remark \ref{remark-2.4}
$$
\frac{\bm{\gamma^1}_{j,n}}{\delta_{j,n}}\ra \int_{\R^2}\til{y}e^U\left\{\bm{a}_j\cdot\nabla U+b_j\clo{U}\right\}=-8\pi\bm{a}_j
$$
which proves \eqref{new3.5}.\\
Finally \eqref{new3.11} follows by \eqref{2.6b}, \eqref{new3.6} and \eqref{new3.5}.

}}
\end{proof}

The next proposition is {\rosso{a refinement of \eqref{new3.5.5}.}}
It relies on a bi-linear form of the Rellich-Pohozaev identity described below. We omit the elementary proof of this identity (see \cite[Proposition 5.5]{O12} for details).
\begin{proposition}
\label{bi-po-prop}
For every $p\in\R^2$, $R>0$, and $f$, $g\in\Co{2}{\clo{\ball{R}{p}}}$, it holds that
\begin{align}
&\int_{\ball{R}{p}}\left\{[(x-p)\cdot\nabla f]\lap g+\lap f [(x-p)\cdot\nabla g]\right\}\nonumber\\
&\qquad=R\int_{\del\ball{R}{p}}\left(2\frac{\del f}{\del \nu}\frac{\del g}{\del \nu}-\nabla f\cdot\nabla g\right).
\label{bi-po}
\end{align}
\end{proposition}
\begin{proposition}\label{newp3.14}
If $\mu_\infty=1$ it holds that
\begin{equation}\label{gamma-0-jn}
\gamma_{j,n}^0=(1-\mn)\left\{\frac{16\pi}{3}b_j+o(1)\right\}+o\left(\ln^{\frac{1}{2}}\right).
\end{equation}
\end{proposition}

\begin{proof}%[Proof of Proposition \ref{newp3.14}]
Putting $p=x_{j,n}$, $f=u_n$, and $g=v_n$ on the left-hand side of \eqref{bi-po}, we have
\begin{align*}
&\int_{\ball{R}{x_{j,n}}}\left\{[(x-x_{j,n})\cdot\nabla u_n]\lap v_n+\lap u_n[(x-x_{j,n})\cdot\nabla v_n]\right\}\\
&=\int_{\ball{R}{x_{j,n}}}\left\{[(x-x_{j,n})\cdot\nabla u_n](-\lambda_n\mn e^{u_n}v_n)-\lambda_n e^{u_n} [(x-x_{j,n})\cdot\nabla v_n]\right\}\\
&=-\int_{\ball{R}{x_{j,n}}}(x-x_{j,n})\cdot\nabla\left(\lambda_n e^{u_n} v_n\right)\\
&\qquad+(1-\mn)\int_{\ball{R}{x_{j,n}}}[(x-x_{j,n})\cdot\nabla u_n](\lambda_n e^{u_n}v_n)\\
&=-\int_{\partial\ball{R}{x_{j,n}}}[(x-x_{j,n})\cdot\nu]\lambda_n e^{u_n} v_n+2\int_{\ball{R}{x_{j,n}}}\lambda_n e^{u_n} v_n\\
&\qquad+(1-\mn)\int_{\ball{R}{x_{j,n}}}[(x-x_{j,n})\cdot\nabla u_n](\lambda_n e^{u_n}v_n)\quad \quad{\rosso{\text{ and by \eqref{new3.11}}}}\\
&=2\gamma^0_{j,n}+(1-\mn)\int_{\ball{R}{x_{j,n}}}[(x-x_{j,n})\cdot\nabla u_n](\lambda_n e^{u_n}v_n)+o(\lambda_n).
\end{align*}
Here, Lemma \ref{newl3.11} implies
\begin{align*}
&\int_{\ball{R}{x_{j,n}}}[(x-x_{j,n})\cdot\nabla u_n](\lambda_n e^{u_n}v_n)=\int_{\ball{\frac{R}{\delta_{j,n}}}{0}}\left(\til{x}\cdot\nabla_{\til{x}} \til{u}_{j,n}\right)e^{\til{u}_{j,n}}\til{v}_{j,n}\\
&\tend\int_{\R^2}\left(\til{x}\cdot\nabla U\right)e^UV_j=\int_{\R^2}e^U\clo{U}\left(\bm{a}_j\cdot\nabla U+b_j\clo{U}\right)-2\int_{\R^2}e^U(\bm{a}_j\cdot\nabla U+b_j\clo{U})\\
&\qquad=\frac{32\pi}{3}b_j,
\end{align*}
(see Lemma \ref{l6.2} for the last integration). Therefore, by (\ref{bi-po}) we obtain
\begin{align}
\gamma^0_{j,n}&=-(1-\mn)\left\{\frac{16\pi}{3}b_j+o(1)\right\}\nonumber\\
&+\frac{R\mn}{2}\int_{\partial\ball{R}{x_{j,n}}}\left(2\frac{\del u_n}{\del \nu}\frac{\del}{\del \nu}\frac{\vn}{\mn}-\nabla u_n\cdot\nabla \frac{v_n}{\mn}\right)+o\left(\lambda_n\right).
\label{rep-gamma}
\end{align}
From \eqref{new3.3} and \eqref{new3.5}, it follows that
\begin{align}
\nabla \frac{v_n(x)}{\mn}&=-\frac{\gamma^0_{j,n}}{2\pi}\cdot\frac{\nu}{R}+4d_j\frac{\bm{a}_j-2(\bm{a}_j\cdot\nu)\nu}{R^2}\lambda_n^\frac{1}{2}+\nabla h_{j,n}(x) +o\left(\lambda_n^\frac{1}{2}\right)\nonumber
\end{align}
on $\del\ball{R}{x_{j,n}}$, where
\begin{align*}
h_{j,n}(x)&=\gamma^0_{j,n}K(x,x_{j,n})+\sum_{1\leq i\leq m, i\not=j}\gamma^0_{i,n}G(x,x_{i,n})\\
&\quad+\bm{\gamma^1}_{j,n}\cdot\nabla_yK(x,x_{j,n})+\sum_{1\leq i\leq m, i\not=j}\bm{\gamma^1}_{i,n}\cdot\nabla_yG(x,x_{i,n}).
\end{align*}
We note that $h_{j,n}$ is harmonic in $\ball{R}{x_{j,n}}$.
Arguing as in Proposition \ref{newp3.3} and  Remark \ref{newr3.4}, we have
\[ \frac{\partial u_n}{\partial x_\alpha}(x)=\sum_{j=1}^m\left\{\sigma_{j,n}^0G_{x_\alpha}(x, x_{j,n})+\bm{\sigma^1}_{j,n}\cdot\nabla_yG_{x_\alpha}(x,x_{j,n})\right\}+o\left(\lambda_n^{\frac{1}{2}}\right) \]
for $\alpha=1,2$ uniformly in $\overline{\Omega}\setminus\cup_{j=1}^mB_R(\kappa_j)$, where
$$
\sigma_{j,n}^0=\int_{\ball{R}{x_{j,n}}}\lambda_n e^{u_n}\tend 8\pi
$$
and $\bm{\sigma^1}_{j,n}=(\sigma_{j,n}^{1,1}, \sigma_{j,n}^{1,2})$ with
\[ \sigma_{j,n}^{1,\alpha}=\int_{B_R(x_{j,n})}(y-x_{j,n})_\alpha\lambda_ne^{u_n}dy.  \]
Similarly to (\ref{new3.5}), inequality (\ref{2.6a}) implies
$$
\frac{\bm{\sigma^1}_{j,n}}{\delta_{j,n}}=\frac{1}{\delta_{j,n}}\int_{\ball{R}{x_{j,n}}}(y-x_{j,n})\ln e^{u_n}\tend\int_{\R^2}\til{y}e^U=\bm{0},
$$
i.e., $\bm{\sigma^1}_{j,n}=o(\delta_{j,n})=o\left(\lambda_n^{\frac{1}{2}}\right)$.  Then we obtain
\begin{equation}
\nabla u_n(x)=-\frac{\sigma_{j,n}^0}{2\pi}\cdot\frac{\nu}{R}+\nabla k_{j,n}(x)+o\left(\lambda_n^\frac{1}{2}\right)\label{u_asympt_sing}
\end{equation}
on $\del\ball{R}{x_{j,n}}$ by (\ref{2.2}), where
$$
k_{j,n}(x)=\sigma_{j,n}^0K(x,x_{j,n})+\sum_{1\leq i\leq m, i\not=j}\sigma_{i,n}^0G(x,x_{i,n}).
$$
These formulae imply
\begin{align*}
&\int_{\partial\ball{R}{x_{j,n}}}2\frac{\del u_n}{\del \nu}\frac{\del}{\del \nu}\frac{\vn}{\mn}=2\int_{\partial\ball{R}{x_{j,n}}}\left(-\frac{\sigma_{j,n}^0}{2\pi}\cdot\frac{1}{R}+\frac{\del k_{j,n}}{\del\nu}\right)\\
&\qquad\qquad\times\left(-\frac{\gamma^0_{j,n}}{2\pi}\cdot\frac{1}{R}+4d_j\frac{-\bm{a}_j\cdot\nu}{R^2}\lambda_n^\frac{1}{2}+\frac{\del h_{j,n}}{\del\nu}\right) +o\left(\lambda_n^\frac{1}{2}\right)\\
&=\frac{\sigma_{j,n}^0\gamma^0_{j,n}}{\pi R}+\frac{4\sigma_{j,n}^0d_j\lambda_n^\frac{1}{2}}{\pi R^3}\int_{\partial\ball{R}{x_{j,n}}}\bm{a}_j\cdot\nu-\frac{\sigma_{j,n}^0}{\pi R}\int_{\partial\ball{R}{x_{j,n}}}\frac{\del h_{j,n}}{\del\nu}\\
&\quad-\frac{\gamma^0_{j,n}}{\pi R}\int_{\partial\ball{R}{x_{j,n}}}\frac{\del k_{j,n}}{\del\nu}-\frac{8d_j\lambda_n^\frac{1}{2}}{R^2}\int_{\partial\ball{R}{x_{j,n}}}(\bm{a}_j\cdot\nu)\frac{\del k_{j,n}}{\del\nu}\\
&\qquad+\int_{\partial\ball{R}{x_{j,n}}}2\frac{\del k_{j,n}}{\del\nu}\frac{\del h_{j,n}}{\del\nu}+o\left(\lambda_n^\frac{1}{2}\right)\\
&=\frac{\sigma_{j,n}^0\gamma^0_{j,n}}{\pi R}-\frac{8d_j\lambda_n^\frac{1}{2}}{R^2}\int_{\partial\ball{R}{x_{j,n}}}(\bm{a}_j\cdot\nu)\frac{\del k_{j,n}}{\del\nu} \\
&\qquad +\int_{\partial\ball{R}{x_{j,n}}}2\frac{\del k_{j,n}}{\del\nu}\frac{\del h_{j,n}}{\del\nu}+o\left(\lambda_n^\frac{1}{2}\right),
\end{align*}
by the divergence formula because $h_{j,n}$ and $k_{j,n}$ are harmonic in $\ball{R}{x_{j,n}}$.\\
Similarly it holds that
\begin{align*}
&\int_{\partial\ball{R}{x_{j,n}}}\nabla u_n\cdot\nabla \frac{v_n}{\mn}=\int_{\partial\ball{R}{x_{j,n}}}\left(-\frac{\sigma_{j,n}^0}{2\pi}\cdot\frac{\nu}{R}+\nabla k_{j,n}\right)\\
&\qquad\qquad{\rosso{\cdot}}\left(-\frac{\gamma^0_{j,n}}{2\pi}\cdot\frac{\nu}{R}+4d_j\frac{\bm{a}_j-2(\bm{a}_j\cdot\nu)\nu}{R^2}\lambda_n^\frac{1}{2}+\nabla h_{j,n}\right) +o\left(\lambda_n^\frac{1}{2}\right)\\
&\quad=\frac{\sigma_{j,n}^0\gamma^0_{j,n}}{2\pi R}+\frac{4d_j\lambda_n^\frac{1}{2}}{R^2}\int_{\partial\ball{R}{x_{j,n}}}\bm{a}_j\cdot\nabla k_{j,n}-\frac{8d_j\lambda_n^\frac{1}{2}}{R^2}\int_{\partial\ball{R}{x_{j,n}}}(\bm{a}_j\cdot\nu)\frac{\del k_{j,n}}{\del\nu}\\
&\qquad+\int_{\partial\ball{R}{x_{j,n}}}\nabla k_{j,n}\cdot\nabla h_{j,n}+o\left(\lambda_n^\frac{1}{2}\right).
\end{align*}
Here, the identity
$$
\int_{\partial\ball{R}{x_{j,n}}}2\frac{\del k_{j,n}}{\del\nu}\frac{\del h_{j,n}}{\del\nu}-\int_{\partial\ball{R}{x_{j,n}}}\nabla k_{j,n}\cdot\nabla h_{j,n}=0
$$
follows again from the bi-linear Pohozaev identity (\ref{bi-po}) because $h_{j,n}$ and $k_{j,n}$ are harmonic. We have also
\begin{align*}
\frac{4d_j\lambda_n^\frac{1}{2}}{R^2}\int_{\partial\ball{R}{x_{j,n}}}\bm{a}_j\cdot\nabla k_{j,n}&=\frac{4d_j\lambda_n^\frac{1}{2}}{R^2}|\partial\ball{R}{x_{j,n}}|\bm{a}_j\cdot\nabla k_{j,n}(x_{j,n})\\
&=\frac{8\pi d_j\lambda_n^\frac{1}{2}}{R}\bm{a}_j\cdot\nabla k_{j,n}(x_{j,n})=o\left(\lambda_n^\frac{1}{2}\right)
\end{align*}
by the mean value theorem for harmonic functions, because, from (\ref{conditionS})
$$
\nabla k_{j,n}(x_{j,n})\tend 8\pi\nabla \left.\left(K(x,\kappa_j)+\sum_{1\leq i\leq m,\,i\not= j}G(x,\kappa_i)\right)\right|_{x=\kappa_j}=0.
$$

Plugging these formulae to \eqref{rep-gamma}, we end up with
$$
\gamma^0_{j,n}=-(1-\mn)\left\{\frac{16\pi}{3}b_j+o(1)\right\}+\frac{\sigma_{j,n}^0\gamma^0_{j,n}\mn}{4\pi}+o\left(\lambda_n^\frac{1}{2}\right),
$$
which means, since $\sigma_{j,n}^0\tend 8\pi$ and $\mn\tend 1$,
\begin{align}
\gamma^0_{j,n}&=\frac{-(1-\mn)\left\{\frac{16\pi}{3}b_j+o(1)\right\}+o\left(\lambda_n^\frac{1}{2}\right)}{1-\frac{\sigma_{j,n}^0\mn}{4\pi}}\nonumber\\
&=(1-\mn)\left\{\frac{16\pi}{3}b_j+o(1)\right\}+o\left(\lambda_n^\frac{1}{2}\right).
\label{sharp_gamma_estimate}
\end{align}
%by $\sigma_{j,n}^0\tend 8\pi$ and $\mn\tend 1$.
\end{proof}

\begin{proposition}\label{newp3.12}
Let $\mu_\infty=1$ and $b_j\not=0$ for some $j\in\{1,\cdots,m\}$. Then it holds that
$$
\mn=1-\frac{3}{2}\frac{1}{\log\ln}+o\left(\frac{1}{\log\ln}\right).
$$
\end{proposition}
\begin{proof}
Combining Propositions \ref{prop:2.9} and \ref{newp3.14} we have
$$
\frac{8\pi b_j+o(1)}{\log\ln}=(1-\mn)\left\{\frac{16\pi}{3}b_j+o(1)\right\}+o\left(\ln^\frac{1}{2}\right).
$$
Then we obtain, since  $b_j\not=0$,
\begin{align*}
\mn&=1-\frac{8\pi b_j+o(1)}{\frac{16\pi}{3}b_j+o(1)}\cdot\frac{1}{\log\ln}+o\left(\ln^\frac{1}{2}\right)\\
&=1-\frac{3}{2}\cdot\frac{1}{\log\ln}+o\left(\frac{1}{\log\ln}\right) .
\end{align*}
\end{proof}
\begin{remark}\label{newr3.13}
By Proposition \ref{newp3.12}, we only have  to consider the cases $b_j=0$ ($j=1,\cdots, m$) to calculate the Morse index of $u_n$ because, {\rosso{by the last proposition, we have that }} $\mn>1$, for $n$ large enough, {\rosso{if}} $b_j\not=0$.
\end{remark}
%Let us introduce some notations.
Let $D=(D_{ij})$ be the diagonal matrix $\mathrm{diag}[d_1,d_1,d_2,d_2,\cdots,d_m,d_m]$, i.e. $D$ takes the diagonal components{\rosso{ $D_{2j-1,2j-1}=D_{2j,2j}=d_j$ }}for $j=1,\cdots,m$.\\ %Then we put
%\begin{equation}\label{DHD}
%\widetilde{\mathrm{Hess}}\,H^m(\ka_1,\cdots,\ka_m)=D\mathrm{Hess}\,H^m(\ka_1,\cdots,\ka_m)D.
%\end{equation}
Let us introduce some notations. Set
\begin{align}
&I_{\alpha\beta}(z_1,z_2,z_3)\nonumber\\
&:=\int_{\del\ball{R}{z_1}}\left\{\frac{\del}{\del\nu_x}G_{x_\alpha}(x,z_2)G_{y_\beta}(x,z_3)-G_{x_\alpha}(x,z_2)\frac{\del}{\del\nu_x}G_{y_\beta}(x,z_3)\right\}d\sigma_x\nonumber\\
  &=\label{I}
\left\{
\begin{array}{ll}
0 &\quad(z_1\not=z_2,\; z_1\not=z_3)\\
\frac{1}{2}R_{x_\alpha x_\beta}(z_1)  &\quad(z_1=z_2=z_3)  \\
G_{x_\alpha y_\beta}(z_1,z_3) &\quad (z_1=z_2,\; z_1\not=z_3) \\
G_{x_\alpha x_\beta}(z_1,z_2) &\quad (z_1\not=z_2,\; z_1=z_3),
\end{array}
\right.
\end{align}
(see \cite[Proposition 2.3]{GOS11}).
\bigskip
{\rosso{We prove a crucial estimate }}for the proof of Theorem \ref{t1}.
\begin{proposition}\label{newp3.16}
If $1-\mn=o\left(\ln^\frac{1}{2}\right)$, there exists $\eta\in\R$ such that
%, an eigenvalue of $\til{\mathrm{Hess}}H^m(\ka_1,\cdots,\ka_m)$ {\rosso{(see \eqref{DHD} for the definition of $\til{\mathrm{Hess}}H^m$)}}, satisfying
\begin{equation}\label{new3.14.8}
\mn=1-48\pi\eta\ln+o(\ln).
\end{equation}
Moreover $\eta$ is an eigenvalue of %the Hessian matrix of $DH^m(\ka_1,\cdots,\ka_m)D$.  
the matrix $D\{\mathrm{Hess}H^m(\ka_1,\cdots,\ka_m)\}D$.
% where $D$ is the diagonal matrix $\mathrm{diag}[d_1,d_1,d_2,d_2,\cdots,d_m,d_m]$ and $d_i$ as in \eqref{2.6b}.
\end{proposition}
\begin{proof}
Since $1-\mn=o\left(\ln^\frac{1}{2}\right)$ it holds that
\begin{equation}\label{new3.15.1}
v_n=-8\pi\ln^\frac{1}{2}\sum_{j=1}^md_j\bm{a}_j\cdot\nabla_y G(x,x_{j,n})+o\left(\ln^\frac{1}{2}\right)
\end{equation}
in $\Co{1}{\clo{\Omega}\backslash \cup_{j=1}^m\ball{R}{\ka_j}}$ {\rosso {by \eqref{asin-v-n}, \eqref{new3.5}, \eqref{gamma-0-jn} and \eqref{2.6b}.}}
%Proposition  \ref{newp3.3}, relation \eqref{new3.5}, and Proposition \ref{newp3.14}.  We have also $b_j=0$ for any $j=1,\cdots,m$ by Proposition \ref{newp3.12}.

Using that
$$
-\lap\frac{\de \un}{\de x_\alpha}=\ln e^{\un}\frac{\de \un}{\de x_\alpha}\quad\text{in $\Omega$},
$$
we have, {\rosso{for $\alpha=1,2$,}}
\begin{align}
\int_{\de B_R(x_{j,n}) } &\left\{ \frac{\de}{\de \nu}\left(\frac{\de \un}{\de x_\alpha}\right) \frac{\vn}{\ln^\frac{1}{2}}-\frac{\de \un}{\de x_\alpha} \frac{\de}{\de \nu}\left(\frac{\vn}{\ln^\frac{1}{2}}\right)\right\}\, d\sigma_x
\nonumber\\
=&\int_{B_R(x_{j,n}) }\left(\Delta \frac{\de \un}{\de x_\alpha} \frac{\vn}{\ln^\frac{1}{2}}-\frac{\de \un}{\de x_\alpha} \Delta\frac{\vn}{\ln^\frac{1}{2}}\right) \, dx\nonumber\\
=&\frac{-1+\mu_n}{\ln^\frac{1}{2}} \int_{B_R(x_{j,n}) }\!\!\!\!\!\!\ln e^{\un}\frac{\de \un}{\de x_\alpha}\vn \, dx\nonumber\\
=&\frac{\mu_n-1}{\ln^\frac{1}{2}\delta_{j,n}} \int_{\ball{\frac{R}{\delta_{j,n}}}{0} }\!\!\!\!\!\!e^{\til{u}_{j,n}}\frac{\de \til{u}_{j,n}}{\de \til{x}_\alpha}\til{v}_{j,n} \, d\til{x}.
 \label{new3.15.5}
\end{align}
Then we obtain, {\rosso{recalling that $b_j=0$,}}
\begin{align}\label{new3.16}
\int_{\ball{\frac{R}{\delta_{j,n}}}{0} }\!\!\!\!\!\!e^{\til{u}_{j,n}}\frac{\de \til{u}_{j,n}}{\de \til{x}_\alpha}\til{v}_{j,n} \, d\til{x}\tend\int_{\R^2}\!\!\!\!\!\!e^{U}U_\alpha \left(\bm{a}_j\cdot\nabla U%+b_j\clo{U}
\right)=\frac{4\pi}{3}a_{j,\alpha}
\end{align}
{\rosso{where $\bm{a}_j=(a_{j,1},a_{j,2})$}},
by \eqref{2.6a} and Lemma \ref{newl3.11} (see Lemma \ref{l6.2} for the last integration).
Then, by \eqref{I}
\begin{align*}
&\sum_{1\leq k\leq m}I_{\alpha\beta}(\kappa_j,\kappa_k,\kappa_l)\\
&=\left\{
\begin{array}{ll}
\frac{1}{2}R_{x_\alpha x_\beta}(\kappa_j)+\displaystyle\sum_{\stackrel{1\leqq k\leqq m}{k\not=j}}G_{x_\alpha x_\beta}(\kappa_j,\kappa_k) , &\quad(j=l)  \\
G_{x_\alpha y_\beta}(\kappa_j,\kappa_l) , &\quad(j\not=l)
\end{array}
\right.
\\
&=H^m_{x_{j,\alpha }x_{l,\beta}}(\kappa_1,\cdots,\kappa_m).
\end{align*}

{\rosso{ Moreover from }}\eqref{new3.15.1} {\rosso{ and \eqref{2.2}}}, it follows that
\begin{align}
\int_{\de B_R(x_{j,n}) } &\left\{ \frac{\de}{\de \nu}\left(\frac{\de \un}{\de x_\alpha}\right) \frac{\vn}{\ln^\frac{1}{2}}-\frac{\de \un}{\de x_\alpha} \frac{\de}{\de \nu}\left(\frac{\vn}{\ln^\frac{1}{2}}\right)\right\}\, dx\\
&\tend\int_{\del\ball{R}{\kappa_j}}\left\{\frac{\del}{\del\nu}8\pi\sum_{k=1}^mG_{x_\alpha}(x,\kappa_k)\cdot (-8\pi)\sum_{l=1}^m d_l\bs{a}_l\cdot\nabla_yG(x,\kappa_l)\right.
\nonumber\\
&\qquad\left.-8\pi\sum_{k=1}^mG_{x_\alpha}(x,\kappa_k)\frac{\del}{\del\nu}(-8\pi)\sum_{l=1}^md_l\bs{a}_l\cdot\nabla_yG(x,\kappa_l)\right\}d\sigma_x
\nonumber\\
&=-64\pi^2\sum_{1\leq k, l\leq m}d_l\sum_{\beta=1}^2a_{l,\beta}I_{\alpha\beta}(\kappa_j,\kappa_k,\kappa_l)
\nonumber\\
&=-64\pi^2\sum_{\stackrel{1\leq l\leq m}{\beta=1,2}}\left(\sum_{1\leq k\leq m} I_{\alpha\beta}(\kappa_j,\kappa_k,\kappa_l)\right)d_la_{l,\beta}\nonumber\\
 %\label{new3.16.5}
&{\rosso{=-64\pi^2 \sum _{  \stackrel{1\leq l\leq m} {\beta=1,2}}   H^m_{x_{j,\alpha }x_{l,\beta}}(\kappa_1,\cdots,\kappa_m)d_la_{l,\beta}  .}}\nonumber
%\left(\mathrm{Hess} \ H^m(\kappa_1,\cdots,\kappa_m)\right)_{l,\beta} d_la_{l,\beta}  .}}\nonumber
\end{align}
%we obtain
%$$
%\eqref{new3.16.5}=-64\pi^2\left.\mathrm{Hess} \ H^m(\kappa_1,\cdots,\kappa_m)\cdot^t(d_1\bs{a}_1,\cdots,d_m\bs{a}_m)\right|_{j,\alpha}.
%$$
Consequently, it follows from \eqref{new3.15.5} {\rosso{ and \eqref{new3.16}}} that
\begin{align*}
&-64\pi^2 {\rosso{    \sum _{  \stackrel{1\leq l\leq m} {\beta=1,2}}  H^m_{x_{j,\alpha }x_{l,\beta}}(\kappa_1,\cdots,\kappa_m) %\left(\mathrm{Hess} \ H^m(\kappa_1,\cdots,\kappa_m)\right)_{l,\beta}
d_la_{l,\beta}          }}+o(1)\\
&\qquad=\frac{\mn-1}{\ln}\left\{\frac{4\pi}{3d_j}a_{j,{\rosso{\alpha}}}+o(1)\right\},
\end{align*}
which is equivalent to
\begin{align}
&\frac{1-\mn}{\ln}\left\{a_{j,{\rosso{\alpha}}}+o(1)\right\}
\nonumber\\
&=48\pi d_j{\rosso{        \sum _{  \stackrel{1\leq l\leq m} {\beta=1,2}}   H^m_{x_{j,\alpha }x_{l,\beta}}(\kappa_1,\cdots,\kappa_m)
%\left(\mathrm{Hess} \ H^m(\kappa_1,\cdots,\kappa_m)\right)_{l,\beta}
d_la_{l,\beta}           }}+o(1)
\nonumber\\
&=48\pi {\rosso{        \sum _{  \stackrel{1\leq l\leq m} {\beta=1,2}}  
\left(D\mathrm{Hess}H^m(\kappa_1,\cdots,\kappa_m)D\right)_{2(j-1)+\alpha, 2(l-1)+\beta
}
%  \left(DH^m(\kappa_1,\cdots,\kappa_m)D\right)_{x_{j,\alpha }x_{l,\beta}}
%\left(\til{\mathrm{Hess}} \ H^m(\kappa_1,\cdots,\kappa_m)\right)_{l,\beta}
a_{l,\beta}           }}+o(1).
\label{new3.16.7}
\end{align}
%{\rosso{where $D$ is the diagonal matrix   $\mathrm{diag}[d_1,d_1,d_2,d_2,\cdots,d_m,d_m]$.     }}\\
By Proposition \ref{newp3.8},{\rosso{ recalling that $b_j=0$ for any $j$}} we have {\rosso{that}} $a_{j,\alpha}\not=0$ {\rosso{for some $j=1,\dots,m$ and $\alpha=1,2$}}.  For such $(j, \alpha)$ it holds that
$$
\frac{1-\mn}{48\pi\ln}\tend\frac{{\rosso{    \sum _{  \stackrel{1\leq l\leq m} {\beta=1,2}} 
\left(D\mathrm{Hess}H^m(\kappa_1,\cdots,\kappa_m)D\right)_{2(j-1)+\alpha, 2(l-1)+\beta}
% \left(DH^m(\kappa_1,\cdots,\kappa_m)D\right)_{x_{j,\alpha }x_{l,\beta}}  
%\left(\til{\mathrm{Hess}} \ H^m(\kappa_1,\cdots,\kappa_m)\right)_{l,\beta}
a_{l,\beta}           }}}{a_{j,\alpha}}=:\eta
$$
and hence \eqref{new3.14.8} {\rosso{follows}}. Then \eqref{new3.16.7} implies
$$
\eta a_{j,\alpha}=
{\rosso{\sum _{  \stackrel{1\leq l\leq m} {\beta=1,2}}  
\left(D\mathrm{Hess}H^m(\kappa_1,\cdots,\kappa_m)D\right)_{2(j-1)+\alpha, 2(l-1)+\beta}
%\left(DH^m(\kappa_1,\cdots,\kappa_m)D\right)_{x_{j,\alpha }x_{l,\beta}}  
 %\left(\til{\mathrm{Hess}} \ H^m(\kappa_1,\cdots,\kappa_m)\right)_{l,\beta}
a_{l,\beta} }}
$$
for any $j$ and $\alpha$.  Thus $\eta$ is an eigenvalue of 
%${\mathrm{Hess}}\left(DH^m(\kappa_1,\cdots,\kappa_m)D\right)$  
$D\mathrm{Hess}H^m(\kappa_1,\cdots,\kappa_m)D$
and the proof is complete.
\end{proof}

\bigskip

We conclude this section with some orthogonality relations between the eigenfunctions. Let $v_n^l$ be the $l$-th eigenfunction of \eqref{autov-n}.  Let $\c^l$, $\a^l$, and $\b^l$ be the associated coefficients $\c$, $\a$, and $\b$, arising in the limit of $\tilde v_n^l$ as $n\rightarrow\infty$ (see Remark {\rosso{\ref{remark-2.4}}}).  Thus we assume the orthogonality in Dirichlet norm,
\begin{equation}\label{ortho}
\int_\Omega\nabla v_n^l\cdot\nabla v_n^{l'}=0\quad \text{if $l\not=l'$.}
\end{equation}
%This relation is reflected to the finite dimensional level through the process of $\lambda\downarrow 0$:

\begin{proposition}\label{newp3.15}
We have the following,
\begin{enumerate}
\item[$i)$]
If $\mu_\infty^l=0$, it holds that
\begin{align}\label{grad-v-nl}
\int_{\Om} \left| \na \vn^l\right|^2\, dx &=\mn^l\left\{8\pi\nor\c^l\nor_{\R^m}^2+o(1)\right\}=-\frac{4\pi\nor\c^l\nor_{\R^m}^2}{\log\ln}+o\left(\frac{1}{\log\ln}\right).
\end{align}
\item[$ii)$]
If $\mu_\infty^l=\mu_\infty^{l'}=0$ and $l\not=l'$, it holds that
$$
\c^l\cdot\c^{l'}=0.
$$
\item[$iii)$]
If $\mu_\infty^l=1$, it holds that
\begin{align*}
\int_{\Om} \left| \na \vn^l\right|^2\, dx &=\frac{4\pi}{3}\left(\nor\a^l\nor_{\R^{2m}}^2+8\nor\b^l\nor_{\R^m}^2\right)+o\left(1\right).
\end{align*}
\item[$iv)$]
If $\mu_\infty^l=\mu_\infty^{l'}=1$ and $l\not=l'$, it holds that
$$
\a^{l}\cdot\a^{l'}+8\b^l\cdot\b^{l'}=0.
$$
\end{enumerate}
\end{proposition}
\begin{proof}
Each limit below is justified by (\ref{2.6a}) and the dominated convergence theorem.
\begin{enumerate}
\item[$i)$]
From {\rosso{\eqref{autov-n}}} we have
$$\intO |\na \vn^l|^2\, dx=-\intO (\Delta \vn^l)  \vn^l\, dx=\mn \ln \intO e^{\un}(\vn^l)^2\, dx$$
so that
\begin{eqnarray*}
& & \frac 1{\mn} \intO |\na \vn^l|^2\, dx=\ln \intO e^{\un}(\vn^l)^2\, dx\\
& & =\sum_{j=1}^m \ln \int_{B_R(x_{j,n})} e^{\un}(\vn^l)^2\, dx+\ln \int_{\Om \setminus \bigcup_{j=1}^mB_R(x_{j,n})}  e^{\un}(\vn^l)^2\, dx\\
& & =\sum_{j=1}^m\int_{B_{\frac R{\d_{j,n}}}(0)} e^{\tu_{j,n}}(\tv_{j,n}^l)^2\, d\til{x}+o(\ln) \\
& & \tend\sum_{j=1}^m\int_{\R^2}e^U (c_j^l)^2\, d\til{x}=8\pi\sum_{j=1}^m(c_j^l)^2=8\pi\nor\c^l\nor_{\R^m}^2.
\end{eqnarray*}
{\rosso{By Proposition \ref{newp3.10} $i)$ follows.}}
\item[$ii)$]
If $l\not=l'$, we have
$$0=\intO \na \vn^l\cdot\na \vn^{l'}\,dx=-\intO \Delta \vn^l\vn^{l'}\,dx= \mn^l \ln\intO e^{\un}\vn^l\vn^{l'}\, dx.$$
Since $\mn^l>0$, it holds that
\begin{align*}
0&=\ln\intO e^{\un}\vn^l\vn^{l'}\, dx\\
&=\sum_{j=1}^m \ln \int_{B_R(x_{j,n})} e^{\un}\vn^l\vn^{l'}\, dx+\ln \int_{\Om \setminus \bigcup_{j=1}^mB_R(x_{j,n})}  e^{\un}\vn^l\vn^{l'}\, dx\\
&=\sum_{j=1}^m \int_{B_{\frac R{\d_{j,n}}}(0)} e^{\tu_{j,n}}\tv_{j,n}^l\tv_{j,n}^{l'}\, d\til{x}+o(\ln)\\
&\tend\sum_{j=1}^m\int_{\R^2}e^U c_j^lc_j^{l'}\, d\til{x}=8\pi \c^l\cdot\c^{l'}.
\end{align*}
\item[$iii)$]
Similarly to the first case it holds that
\begin{align*}
&\frac 1{\mn^l} \intO |\na \vn^l|^2\, dx\\
&=\ln \intO e^{\un}(\vn^l)^2\, dx=\sum_{j=1}^m\int_{B_{\frac R{\d_{j,n}}}(0)} e^{\tu_{j,n}}(\tv_{j,n}^l)^2\, d\til{x}+o(\ln)\\
&\tend\sum_{j=1}^m\int_{\R^2}e^U (\bm{a}_j^l\cdot\nabla U+b_j^l\clo{U})^2\, d\til{x}=\frac{4\pi}{3}\nor\a^l\nor_{\R^{2m}}^2+\frac{32\pi}{3}\nor\b^l\nor_{\R^m}^2,
\end{align*}
(see Lemma \ref{l6.2} for the last integration).
\item[$iv)$]
The proof is similar to the second and the third cases.
\end{enumerate}
\end{proof}

%%%%%%%%%%%%%%%%%%%%%%%%%%%%%%%%%SEZIONE
\sezione{Estimate of eigenvalues}\label{s4}
%%%%%%%%%%%%%%%%%%%%%%%%%%%%%%%%%%%%%%%%%%%%%%%%%%%%%%%%%%%%%%
In this section we provide the estimates \eqref{10}, \eqref{10_1} and \eqref{10_2}.
\sottosezione{Estimates of $\mn^1, \dots,\mn^{m}$}\label{sott-1}
\noindent
The asymptotic behavior of the eigenvalues of (\ref{autov-n}) is estimated inductively. Let us start to study $\mn^1$. \\[.5cm]
Take a cut-off function $\xi \in C^{\infty}_0 ([0,+\infty))$ %, $\xi_r(0)=0$
 satisfying
$$\xi(r):=\left\{
\begin{array}{ll}
1 & \hbox{ if } 0\leq r\leq 1\\
0&\hbox{ if } r\geq 2
\end{array}
\right.
\quad 0\leq \xi(r)\leq 1,
$$
and set
\begin{equation}\label{xi-n}
\xi_n(x):=\xi\left( \frac{|x-x_{i,n}|}R\right).
\end{equation}
for some $i\in \{1,\dots, m\}$.
\begin{proposition}\label{p3.3}
The first eigenvalue $\mn^1$ of \eqref{autov-n} satisfies
$$\mn^1\tend 0\quad \text{ as } n\to \infty.$$
\end{proposition}
\begin{proof}
By the classical Rayleigh-Ritz variational formula, it holds that
$$\mn^1=
\inf_{\substack{
v\in H^1_0(\Om) \\
v\not\equiv 0}}
\frac{\int_\Om |\na v|^2\, dx}{\ln\int_\Om e^{\un} v^2\, dx}\leq \frac{\int_{\Om}|\na \vn|^2\, dx}{\ln\int_\Om e^{\un}\vn^2\, dx}
$$
for any $\vn\in H^1_0(\Om)$. Let $\vn:=\xi_n \un$.

Using the asymptotic behavior of $\un$, we have
\begin{align}
\int_{\Om}|\na \vn|^2\, dx&=\intO |\na (\xi_n \un)|^2\, dx= \int_{    B_R(x_{i,n})}\!\!\!\! |\na \un|^2\, dx +O(1)\nonumber\\
&=\int_{\de  B_R(x_{i,n})} \dfrac{\de}{\de \nu}\un\cdot \un\, d\sigma -\int_{    B_R(x_{i,n})}\!\!\!\!\Delta \un \un \, dx+O(1)\nonumber\\
&=\ln\int_{    B_R(x_{i,n})}\!\!\!e^{\un}\un \, dx+O(1)\nonumber\\
&=\un(x_{i,n}) \int_{B_{\frac R{\d_{i,n}}}(0)}\!\!\!\!\! e^{\tu_{i,n}}\, d\til{x}+\int_{B_{\frac R{\d_{i,n}}}(0) }  \!\!\!\!\!   e^{\tu_{i,n}} \tu_{i,n}\, d\til{x}+O(1)\nonumber\\
&=\un(x_{i,n}) \left\{8\pi+o(1)\right\}+O(1)\nonumber\\
&=-16\pi\log\ln+o(\log\ln)\label{3.4a}
\end{align}
as $n\to \infty$.  Also, we have
\begin{align}\label{3.4e}
\ln\int_\Om e^{\un}\vn^2\, dx&=\ln \intO e^{\un}\xi_n^2 \un^2 \, dx= \ln \int_{    B_R(x_{i,n})} e^{\un}\un^2\, dx+O\big( \ln\big)\nonumber\\
&= \int_{B_{\frac R{\d_{i,n}}}(0)}\!\!\!\!\! e^{\tu_{i,n}}\tu_{i,n}^2\, d\til{x}+2 \un(x_{i,n})\int_{B_{\frac R{\d_{i,n}}}(0)}\!\!\!\!\! e^{\tu_{i,n}}\tu_{i,n}\, d\til{x}\nonumber\\
&\quad+\Big(\un(x_{i,n})\Big)^2\int_{B_{\frac R{\d_{i,n}}}(0)}\!\!\!\!\! e^{\tu_{i,n}}\, d\til{x}+o(1)\nonumber\\
&=\left\{8\pi +o(1)\right\}\Big(\un(x_{i,n})\Big)^2+O\left(\un(x_{i,n})\right)
\nonumber\\
&=32\pi\Big(\log\ln\Big)^2+o\left(\left(\log\ln\right)^2\right).
\end{align}

From \eqref{3.4a} and \eqref{3.4e} we deduce
\begin{align*}
0\leq \mn^1&\leq\frac {-16\pi\log\ln+o(\log\ln)}{32\pi\Big(\log\ln\Big)^2+o\left(\left(\log\ln\right)^2\right)}=-\frac{1}{2\log\ln} +o\left(\frac{1}{\log\ln}\right)\tend 0
\end{align*}
and this concludes the proof.
\end{proof}
\begin{remark}
By Proposition \ref{p3.3} there exists $\c^1\in \R^m\setminus\{0\}$ such that
$$
\tv_{j,n}^1\ra c_j^1\quad \text{ in }C^{2,\alpha}_{loc} (\R^2),
$$
by Remark {\rosso{\ref{remark-2.4}}} and Proposition \ref{newp3.8}.
\end{remark}
%%%%%%%%%%%%%%%%%%%%%%%%%%%%%%%
%%%%%%%%%%%%%%%%%%%%%%%%%%%%%%%
%\sottosezione{$k$-th eigenvalue ($k=2,\dots,m$)}
%\subsection{$k$-th eigenvalue ($k=2,\dots,m$).}
In order to prove the estimate for $\mn^2,\dots,\mn^m$ we will proceed by induction.
\begin{proposition}\label{p3.6}
It holds that $\mn^k\to 0$ for any $k=2,\dots,m$.
\end{proposition}
\begin{proof}
Given $k=2,\cdots, m$, we suppose the $l$-th eigenvalue $\mn^l$ tends to 0 as $n\to +\infty$ for any $l\in\{1,\dots,k-1\}$.  Now we shall show $\mn^k\to 0$. Recall that $v_n^l$ denotes the $l$-th eigenfunction of (\ref{autov-n}).  By the above inductive hypothesis, there exist $\c^1,\dots,\c^{k-1}\in \R^m\setminus\{0\}$ and a sub-sequence of $\vn^l$ (denoted by the same symbol) such that $\tv_{j,n}^l\ra c_j^l$ in $C^{2,\alpha}_{loc} (\R^2)$ for any $l=1,\dots,k-1$ and $j=1,\dots,m$.
Moreover, Remark \ref{newr3.5} implies
\begin{equation}\label{conv-zero}
\vn^l\tend 0 \quad \text{in $\Co{1}{\clo{\Omega}\backslash\cup_{j=1}^m\ball{R}{\kappa_j}}$}
\end{equation}
for $l\in \{1,\dots,k-1\}$ as $n\to \infty$.

From the variational characterization of the eigenvalues we have
$$
\mn^k= \inf_{\substack{
v\in H^1_0(\Omega)\,,\, v\neq 0\\
v\perp \mathrm{span} \{\vn^1,\dots,\vn^{k-1}\}}}
\frac{\intO |\na v|^2 \, dx}{\ln \intO e^{\un} v^2 \, dx}.$$

We take the cut-off function $\xi \in C^{\infty}_0 ([0,+\infty))$ and define $\xi_n(x)$ by \eqref{xi-n}, with $i\in \{1,\dots, m\}$ determined later. Put
$$\vn:=\xi_n \un- s_n^1\vn^1-\dots-s_n^{k-1}\vn^{k-1}$$
with
\begin{equation}\label{3.13.5}
s_n^l:=\frac{\intO \na\left( \xi_n\un\right)\cdot \na \vn^l \,dx}{\intO |\na \vn^l|^2\, dx}.
\end{equation}
Then it follows that $\vn \perp \mathrm{span}\{\vn^1,\dots,\vn^{k-1}\}$ under the Dirichlet norm.

First, we show the following,

\bigskip

\noindent {\bf Claim:} It holds that
\begin{equation}\label{claim}
s_n^l=-\frac{2c_i^l}{\nor\c^l\nor_{\R^m}^2}\log\ln+o(\log\ln) \quad \hbox{ as }n\to +\infty.
\end{equation}
for $l=1,\dots,k-1$.
\begin{proof}[Proof of the Claim]
%By Proposition \ref{newp3.15} we {\rosso{know the }} behavior of $\intO |\na \vn^l|^2\, dx$ for $l=1,\dots,k-1$ as $n\rightarrow\infty$.
By \eqref{autov-n}, it holds that
\begin{align*}
&\intO \na( \xi_n\un)\cdot \na \vn^l\, dx=-\intO \xi_n\un\Delta \vn^l\,dx=\mn^l \ln\intO e^{\un}\vn^l\xi_n \un\,dx\\
&=\mn^l \ln \int_{    B_R(x_{i,n})} \!\!\!\!\!e^{\un}\vn^l\un\, dx+\mn^l\ln\int_{\Om \setminus B_R(x_{i,n})}\!\!\!\!\! e^{\un}\vn^l \xi_n\un \, dx\\
&=\mn^l \un(x_{i,n})\int_{B_{\frac R{\d_{i,n}}}(0)}\!\!\!\!\! e^{\tu_{i,n}}\tv_{i,n}^l\, d\til{x}+\mn^l \int_{B_{\frac R{\d_{i,n}}}(0) }  \!\!\!\!\!   e^{\tu_{i,n}} \tv_{i,n}^l \tu_{i,n}\, d\til{x} +o\Big( \mn^l\ln\Big)  \\
&=\left\{-\frac{1}{2\log\ln}+o\left(\frac{1}{\log\ln}\right)\right\}\left\{-2\log\ln+O(1)\right\}\left\{8\pi c_{\rosso{i}}^l+o(1)\right\}+o(1)\\
&=8\pi c_i^l+o(1).
\end{align*}
Then (\ref{3.13.5}) {\rosso{ and $i)$ of Proposition \ref{newp3.15} imply}}
$$
s_n^l=\frac{8\pi c_i^l+o(1)}{-\frac{4\pi}{\log\ln}\nor\c^l\nor_{\R^m}^2+o\left(\frac{1}{\log\ln}\right)}
$$
and the proof of the claim is complete.
\end{proof}

\bigskip

For $\vn=\xi_n \un -s_n^1\vn^1-\dots-s_n^{k-1}\vn^{k-1}$, we have
$$\mn^l\leq \frac{\intO |\na \vn|^2\, dx}{\ln \intO e^{\un}\vn^2\,dx}.$$
Moreover, using \eqref{3.13.5}, we get
\begin{align}\label{3.4.0}
&\intO |\na \vn|^2\, dx=\intO |\na \big(\xi_n \un -s_n^1\vn^1-\dots-s_n^{k-1}\vn^{k-1}\big)|^2\, dx\\
&=\intO |\na (\xi_n \un)|^2\, dx-2\sum_{l=1}^{k-1} s_n^l \intO \na (\xi_n \un)\cdot \na \vn^l\, dx+\sum_{l=1}^{k-1} \big(s_n^l\big)^2 \intO |\na  \vn^l|^2\, dx\nonumber\\
&=\intO |\na (\xi_n \un)|^2\, dx-\sum_{l=1}^{k-1} \big(s_n^l\big)^2 \intO |\na  \vn^l|^2\, dx.\nonumber
\end{align}
%For the behavior of $\intO |\na (\xi_n \un)|^2\, dx$ we shall use \eqref{3.4a}.
{\rosso{By}} \eqref{claim} and Proposition \ref{newp3.15} {\rosso{we get}}
\begin{align*}
&\big( s_n^l\big)^2\intO |\na \vn^l|^2\, dx\\
&=\left( -\frac{2 c_i^l}{\nor\c^l\nor_{\R^m}^2}\log\ln+o(\log\ln)\right)^2\left(-\frac{4\pi}{\log\ln}\nor\c^l\nor_{\R^m}^2+o\left(\frac{1}{\log\ln}\right)\right)\\
&=-\frac{16\pi (c_i^l)^2}{\nor\c^l\nor_{\R^m}^2}\log\ln+o(\log\ln).
\end{align*}
{\rosso{Then \eqref{3.4.0} becomes, using \eqref{3.4a}}}
\begin{equation}
\intO |\na \vn|^2\, dx
=-16\pi \left\{1-\sum_{l=1}^{k-1}\frac{  (c_i^l)^2}{\nor \c^l\nor_{\R^m}^2}\right\}\log\ln+o(\log\ln).
 \label{eqn:3.8}
\end{equation}
%by \eqref{3.4.0}.

Next, we have
\begin{align}\label{3.4d}
\ln\intO e^{\un}\vn^2\, dx&=\ln \intO e^{\un}\xi_n^2 \un^2 \, dx- 2\sum_{l=1}^{k-1} \ln s_n^l \intO e^{\un}\xi_n \un\vn^l\,dx\nonumber\\
&+\sum_{ l,l'=1} ^{k-1} \ln s_n^l s_n^{l'}\intO e^{\un}\vn^l\vn^{l'}\,dx.
\end{align}
Here we have
\begin{align*}
&\sum_{l=1}^{k-1}\ln s_n^l\int_\Omega e^{u_n}\xi_n u_nv_n^ldx=-\sum_{l=1}^{k-1}s_n^l\int_\Omega \lap\left(\frac{v_n^l}{\mn^l}\right)\xi_n u_ndx\\
&\qquad=\sum_{l=1}^{k-1}\frac{s_n^l}{\mn^l}\int_\Omega \nabla v_n^l\cdot\nabla (\xi_n u_n)dx=\sum_{l=1}^{k-1}\frac{(s_n^l)^2}{\mn^l}\int_\Omega |\nabla v_n^l|^2dx,
\end{align*}
{\rosso{and}}
\begin{align*}
&\sum_{ l,l'=1} ^{k-1} \ln s_n^l s_n^{l'}\intO e^{\un}\vn^l\vn^{l'}\,dx=-\sum_{ l,l'=1} ^{k-1} \frac{s_n^l s_n^{l'}}{\mn^l}\intO \lap \vn^l\cdot \vn^{l'}\,dx
\\
&\qquad=\sum_{ l,l'=1} ^{k-1} \frac{s_n^l s_n^{l'}}{\mn^l}\intO \nabla\vn^l\cdot \nabla\vn^{l'}\,dx=\sum_{ l=1} ^{k-1} \frac{(s_n^l)^2}{\mn^l}\intO |\nabla\vn^l|^2\,dx,
\end{align*}
using the assumption \eqref{ortho}. Therefore, it holds that
\begin{equation}\label{3.9a}
\ln\intO e^{\un}\vn^2\, dx=\ln \intO e^{\un}\xi_n^2 \un^2 \, dx-\sum_{ l=1} ^{k-1} \frac{(s_n^l)^2}{\mn^l}\intO |\nabla\vn^l|^2\,dx.
\end{equation}
%Concerning $\ln \intO e^{\un}\xi_n^2 \un^2 \, dx$, relation \eqref{3.4e} is available.
It holds also that
\begin{align*}
&\frac{(s_n^l)^2}{\mn^l}\intO |\nabla\vn^l|^2\,dx\\
&=\frac{\left\{ -\frac{2 c_i^l}{\nor\c^l\nor_{\R^m}^2}\log\ln+o(\log\ln)\right\}^2}{-\frac{1}{2\log\ln}+o\left(\frac{1}{\log\ln}\right)}\left\{-\frac{4\pi}{\log\ln}\nor\c^l\nor_{\R^m}^2+o\left(\frac{1}{\log\ln}\right)\right\}\\
&=\frac{32\pi(c_i^l)^2}{\nor\c^l\nor_{\R^m}^2}\left(\log\ln\right)^2+o\left(\left(\log\ln\right)^2\right).
\end{align*}
{\rosso{Then \eqref{3.9a} becomes, by \eqref{3.4e}}}
\begin{equation}
\ln \intO e^{\un}\vn^2\, dx =32\pi \left\{ 1-\sum_{l=1}^{k-1}\frac{(c_i^l)^2}{  \nor \c^l\nor_{\R^m}^2}\right\}\left(\log\ln\right)^2+o\left(\left(\log\ln\right)^2\right).
 \label{eqn:3.10}
\end{equation}

Here we choose $i\in \{1,\dots,m\}$ satisfying
\begin{equation}\label{3.5}
e_i=(0,\dots,1,0,\dots,0)\notin\C(k-1)
\end{equation}
where
$$\C(k-1):=\mathrm{span}\{\c^1,\dots,\c^{k-1}\}.$$
This is possible because $\C(k-1)$ is a $(k-1)$-dimensional subspace of $\R^m=\mathrm{span}\{e_1,\dots,e_m\}$ (see Proposition \ref{newp3.15}) and $m-(k-1)\geq 1$ by $k\leq m$.  Then, denoting by $\mathrm{proj}_{S}v$ the projection of the vector $v$ on {\rosso{a}} subspace $S$ of $\R^m$, we have, by \eqref{3.5}, $\mathrm{proj}_{\C(k-1)^{\perp}}e_i\neq 0$. Hence it follows that
\begin{align}\label{3.5a}
{\rosso{0<}}|\mathrm{proj}_{\C(k-1)^{\perp}}e_i|^2&=|e_i|^2-|\mathrm{proj}_{\C(k-1)}e_i|^2\nonumber\\
&=1-\left\{\sqrt{\sum_{l=1}^{k-1}\left(e_i\cdot \frac{\c^l}{\nor \c^l\nor_{\R^m}}\right)^2}\right\}^2=1-\sum_{l=1}^{k-1}\frac{(c_i^l)^2}{\nor \c^l\nor _{\R^m}^2}
\end{align}
which implies,  by (\ref{eqn:3.8}) and (\ref{eqn:3.10}),
\begin{align*}
0\leq \mn^k&\leq \frac{-16\pi\left\{ 1- \sum_{l=1}^{k-1}\frac{(c_i^l)^2}{\nor \c^l\nor_{\R^m}^2}\right\}\log\ln+o(\log\ln)}{32\pi\left\{1-\sum_{l=1}^{k-1}\frac{(c_i^l)^2}{\nor \c^l\nor_{\R^m}^2}\right\}\left(\log\ln\right)^2+o \left(\left(\log\ln\right)^2\right)}\\
&=-\frac{1}{2\log\ln}+o\left(\frac{1}{\log\ln}\right)\to 0.
\end{align*}
%by (\ref{eqn:3.8}) and (\ref{eqn:3.10}).
\end{proof}
{\rosso{
\begin{remark}
Propositions \ref{p3.6} and \ref{newp3.10} imply (\ref{10}).
\end{remark}}}
%%%%%%%%%%%%%%%%%%%%%%%%%%%%%%%
%%%%%%%%%%%%%%%%%%%%%%%%%%%%%%%
\sottosezione{Estimates of $\mn^{m+1},\dots,\mn^{3m}$}
Similarly to Section \ref{sott-1} we start by considering $\mu_n^k$ for $k=m+1$.

\begin{proposition}\label{newp4.3} The $m+1$-th eigenvalue $\mn^{m+1}$ of \eqref{autov-n} satisfies
$$0\leq\mn^{m+1}\leq 1+O(\ln)\quad \text{ as } n\to \infty.$$
\end{proposition}
\begin{proof}
We take the cut-off function $\xi_n$ defined by \eqref{xi-n} and put
$$
\vn:=\xi_n \frac{\de \un}{\de x_\alpha}- s_n^1\vn^1-\dots-s_n^{m}\vn^{m}
$$
with
\begin{equation*}
s_n^l:=\frac{\intO \na\left( \xi_n\frac{\de \un}{\de x_\alpha}\right)\cdot \na \vn^l \,dx}{\intO |\na \vn^l|^2\, dx}, \quad l=1,\cdots, m.
\end{equation*}
Then it follows that $\vn \perp \mathrm{span}\{\vn^1,\dots,\vn^{m}\}$.

\bigskip

\noindent {\bf Claim:} It holds that
\begin{equation}\label{claim2}
s_n^l=o(1) \quad \hbox{ as }n\to +\infty.
\end{equation}
for $l=1,\dots,m$.
\begin{proof}[Proof of the Claim]
%First, Proposition \ref{newp3.15} is available for $\intO |\na \vn^l|^2$ with $l=1,\cdots, m$.

By \eqref{autov-n} we have
\begin{align*}
&\intO \na\left( \xi_n\frac{\de \un}{\de x_\alpha}\right)\cdot \na \vn^l\, dx=-\intO \xi_n\frac{\de \un}{\de x_\alpha}\Delta \vn^l\,dx=\mn^l \ln\intO e^{\un}\vn^l\xi_n\frac{\de \un}{\de x_\alpha}\,dx\\
&=\mn^l \ln \int_{    B_R(x_{i,n})} \!\!\!\!\!e^{\un}\vn^l\frac{\de \un}{\de x_\alpha}\, dx+o(\mn^l).
\end{align*}
It holds that, {\rosso{ similarly to the proof of Proposition \ref{newp3.16}}},
\begin{align*}
O(\mn^l)&=\int_{\de\ball{R}{x_{i,n}}} \left\{\frac{\de}{\de\nu}\left(\frac{\de\un}{\de x_\alpha}\right)\cdot\vn^l-\frac{\de\un}{\de x_\alpha}\frac{\de}{\de\nu}\vn^l\right\}\quad \quad {\rosso{\text{ by \eqref{new3.4}}}}\\
&=\int_{\ball{R}{x_{i,n}}} \left\{\lap\left(\frac{\de\un}{\de x_\alpha}\right)\cdot\vn^l-\frac{\de\un}{\de x_\alpha}\lap\vn^l\right\}\\
&=\int_{\ball{R}{x_{i,n}}} \left\{-\ln e^{\un}\frac{\de\un}{\de x_\alpha}\vn^l+\frac{\de\un}{\de x_\alpha}\ln\mn^l e^{\un}\vn^l\right\}\\
&=(\mn^l-1)\int_{\ball{R}{x_{i,n}}}\ln e^{\un}\vn^l\frac{\de\un}{\de x_\alpha}
\end{align*}
which implies
$$
\int_{\ball{R}{x_{i,n}}}\ln e^{\un}\vn^l\frac{\de\un}{\de x_\alpha}=\frac{O(\mn^l)}{\mn^l-1}=O(\mn^l)
$$
{\rosso{since $\mn^l\ra 0$ for any $l=1,\dots,m$.}}
Consequently, we obtain, {\rosso{by $i)$ of Proposition \ref{newp3.15}}},
$$
s_n^l=-\frac{O\left((\mn^l)^2\right)+o(\mn^l)}{\mn^l\left\{8\pi\nor\c^l\nor_{\R^m}^2+o(1)\right\}}=O(\mn^l)=o(1)
\quad {\rosso{\text{ by $\nor\c^l\nor_{\R^m}\not=0$.}}}
$$
\end{proof}

\bigskip

It holds that
$$
0\leq\mn^{m+1}\leq\frac{\intO |\na v_n|^2 \, dx}{\ln \intO e^{\un} v_n^2 \, dx}.
$$
{\rosso{Then, arguing as in \eqref{3.4.0}, from \eqref{grad-v-nl} and \eqref{claim2}, we have}}
\begin{align}
\ln\intO e^{\un}\vn^2\, dx&=\ln \intO e^{\un}\left(\xi_n\frac{\de\un}{\de x_\alpha}\right)^2 \, dx-\sum_{ l=1} ^{m} \frac{(s_n^l)^2}{\mn^l}\intO |\nabla\vn^l|^2\,dx\nonumber\\
&=\int_{\ball{R}{x_{i,n}}}\ln e^{u_n}\left(\frac{\de \un}{\de x_\alpha}\right)^2\, dx+o(1).\label{**}
\end{align}
Then it follows that
\begin{align*}
\intO |\na \vn|^2\, dx&=\intO \left|\na \left(\xi_n \frac{\de \un}{\de x_\alpha}\right)\right|^2\, dx-\sum_{l=1}^{m} \big(s_n^l\big)^2 \intO |\na  \vn^l|^2\, dx\\
&=\int_{\ball{R}{x_{i,n}}} \left|\na \left(\frac{\de \un}{\de x_\alpha}\right)\right|^2\, dx+O(1)\\
&=\int_{\ball{R}{x_{i,n}}}\ln e^{u_n}\left(\frac{\de \un}{\de x_\alpha}\right)^2\, dx+O(1)\quad\quad {\rosso{\text{ by \eqref{**}}}}\\
&=\int_{\Omega}\ln e^{u_n}{\rosso{\vn^2}}\, dx+O(1).
\end{align*}
Here we have
\begin{align*}
&\ln \int_{B_R(x_{i,n})}e^{\un}\left(\frac{\de\un}{\de x_\alpha}\right)^2dx= \int_{B_{\frac R{\d_{i,n}}}(0)}\!\!\!\!\! e^{\tu_{i,n}}\frac{1}{\delta_{i,n}^2}\left(\frac{\de\til{u}_{i,n}}{\de \til{x}_\alpha}\right)^2\, d\til{x}\nonumber\\
&=\frac{1}{d_i^2\ln}\left(\int_{\R^2}e^UU_\alpha^2+o(1)\right)=\frac{4\pi}{3d_i^2\ln}\Big(1+o(1)\Big).
\end{align*}
Therefore, it holds that,{\rosso{ by \eqref{**}}}
\begin{align*}
0\leq \mn^{m+1}&\leq  {\rosso{ \frac{\intO |\na\vn|^2\, dx}{\ln\intO e^{\un}\vn^2\, dx}=}}    \frac{\ln\intO e^{\un}\vn^2\, dx+O(1)}{\ln\intO e^{\un}\vn^2\, dx}=1+\frac{O(1)}{\ln \int_{B_R(x_{i,n})}e^{\un}\left(\frac{\de\un}{\de x_\alpha}\right)^2+o(1)}\\
&=1+\frac{O(1)}{\frac{4\pi}{3d_i^2\ln}\Big(1+o(1)\Big)+o(1)}=1+O(\ln)
\end{align*}
{\rosso{since}} $d_i\not=0$.
\end{proof}
\begin{proposition}\label{newp.4.4}
We have $\mu^{m+1}_\infty=1$ and there exists an eigenvalue $\til{\eta}^{m+1}$ of 
%{\rosso{ the Hessian matrix of $DH^m(\ka_1,\cdots,\ka_m)D$}} 
the matrix $D\{\mathrm{Hess}H^m(\ka_1,\cdots,\ka_m)\}D$
 satisfying

$$
\mn^{m+1}=1-48\pi\til{\eta}^{m+1}\ln+o(\ln).
$$
\end{proposition}
\begin{proof}
%Regarding Proposition \ref{newp4.3},
{\rosso{First }}we show that $\mu^{m+1}_\infty<1$ is impossible.  In fact, if $\mu^{m+1}_\infty<1$ {\rosso{ by Corollary \ref{corollary-2.12-bis} we have that }} $\mu^{m+1}_\infty=0$. Proposition \ref{newp3.15} implies the existence of $\c^{m+1}\in\R^m$ such that
$$
(V_1^{m+1},\cdots,V_m^{m+1})=\c^{m+1}\not=0
$$
and
$$
\c^l\cdot \c^{m+1}=0\quad\text{for every $l\in\{1,\cdots,m\}$}.
$$
This is impossible because $\mathrm{span}\{\c^1,\cdots,\c^m\}=\R^m$. Therefore, {\rosso{by Proposition \ref{newp4.3}}}, we derive that $\mu_\infty^{m+1}=1$, and then Proposition \ref{newp3.16} {\rosso{gives the claim.}}
\end{proof}
%%%%%%%%%%%%%%%%%%%%%%%%%%%%%%%
%%%%%%%%%%%%%%%%%%%%%%%%%%%%%%%
%\sottosezione{$k$-th eigenvalue ($k=m+2,\dots,3m$)}
%\subsection{$k$-th eigenvalue ($k=2,\dots,m$).}
\begin{proposition}\label{newp4.5}
It holds that
\begin{equation}\label{new4.2}
0\leq \mn^k\leq 1+O(\ln)\quad\text{for $k=m+2,\cdots,3m$.}
\end{equation}
\end{proposition}
\begin{proof}
We prove \eqref{new4.2} by induction.
Let us assume that \eqref{new4.2} holds for $m+1\leq l\leq k-1$. %  we shall prove \eqref{new4.2}.  Under this induction hypothesis, we have
{\rosso{So, by Proposition \ref{newp3.16},}}  we have
\begin{equation}\label{new3.15}
\mn^l=1-48\pi\til{\eta}^{l}\ln+o(\ln)
\end{equation}
for $m+1\leq l\leq k-1$, %similarly to $k=m+1$,
where $\til{\eta}^l$ is an eigenvalue of 
%{\rosso{the Hessian matrix of $D\mathrm{Hess}H^m(k_1,\dots,k_m)D$}} 
the matrix $D\mathrm{Hess}H^m(k_1,\dots,k_m)D$.
 Then, from Proposition \ref{newp3.12}, it follows that $\b^l=\bm{0}$, for $m+1\leq l\leq k-1$,
{\rosso{and so
\begin{equation}\label{***}
\til{v}_{i,n}^l\to a_i^l\cdot\na U\quad\quad \text{ in }C^2_{loc}(\R^2).
\end{equation}}}
Now we take the cut-off function $\xi_n$ defined in \eqref{xi-n} and let
$$
\vn:=\xi_n \frac{\de \un}{\de x_\alpha}- s_n^1\vn^1-\dots-s_n^{k-1}\vn^{k-1}
$$
with
\begin{equation*}
s_n^l:=\frac{\intO \na\left( \xi_n\frac{\de \un}{\de x_\alpha}\right)\cdot \na \vn^l \,dx}{\intO |\na \vn^l|^2\, dx}.
\end{equation*}
Here the indices $i\in\{1,\cdots,m\}$ and $\alpha\in\{1,2\}$ {\rosso{will be}} chosen later. In any case it follows that $\vn \perp \mathrm{span}\{\vn^1,\dots,\vn^{k-1}\}$.
{\rosso{By Proposition \ref{newp4.3} we also have $s_n^l=o(1)$ for $1\leq l\leq m$.
%Since the behavior of $s_n^l$ for $1\leq l\leq m$ in the proof of Proposition \ref{newp4.3} is available, here we prepare the following:
Then we have}} the following

\bigskip

\noindent {\bf Claim:}  It holds that
\begin{equation}\label{claim3}
s_n^l=\left(\frac{a_{i,\alpha}^l}{d_i\nor\a^l\nor_{\R^{2m}}^2}+o(1)\right)\frac{1}{\ln^\frac{1}{2}}\quad \hbox{ as }n\to \infty
\end{equation}
for $l=m+1,\dots,k-1$.
\begin{proof}[Proof of the Claim]
We have
\begin{align*}
&\intO \na\left( \xi_n\frac{\de \un}{\de x_\alpha}\right)\cdot \na \vn^l\, dx=\mn^l \ln\intO e^{\un}\vn^l\xi_n\frac{\de \un}{\de x_\alpha}\,dx\\
&=\mn^l \ln \int_{    B_R(x_{i,n})} \!\!\!\!\!e^{\un}\vn^l\frac{\de \un}{\de x_\alpha}\, dx+o\left(1\right)
\end{align*}
and
\begin{align*}
\int_{\ball{R}{x_{i,n}}}\ln e^{\un}v_n^l\left(\frac{\de \un}{\de x_\alpha}\right)&=\frac{1}{\delta_{i,n}}\int_{\ball{\frac{R}{\delta_{i,n}}}{0}}e^{\til{u}_{i,n}}\left(\frac{\de \til{u}_{i,n}}{\de \til{x}_\alpha}\right) \til{v}_{i,n}^l\\
&=\frac{1}{\delta_{i,n}}\left\{\int_{\R^2}e^UU_\alpha\left(\bm{a}_i^l\cdot\nabla U\right)+o(1)\right\}\\
&=\frac{1}{\ln^\frac{1}{2}}\left(\frac{4\pi}{3}\frac{a_{i,\alpha}^l}{d_i}+o(1)\right).
\end{align*}
Then {\rosso{$iii)$ of }}Proposition \ref{newp3.15} implies
\begin{align*}
s_n^l&=\frac{\frac{\mn^l}{\ln^\frac{1}{2}}\left(\frac{4\pi}{3}\frac{a_{i,\alpha}^l}{d_i}+o(1)\right)+o\left(1\right)}{\frac{4\pi}{3}\nor\a^l\nor_{\R^{2m}}^2+o\left(1\right)}=\left(\frac{a^l_{i,\alpha}}{d_i\nor\a^l\nor_{\R^{2m}}^2}+o(1)\right)\frac{1}{\ln^\frac{1}{2}}.
\end{align*}
\end{proof}

It holds that
$$
0\leq \mn^k\leq\frac{\intO |\na v_n|^2 \, dx}{\ln \intO e^{\un} \vn^2 \, dx}.
$$
Here we have,  by \eqref{**},
\begin{align*}
\intO&\ln e^{\un}\vn^2 dx=\int_{\Omega}\ln e^{u_n}\left(\xi_n\frac{\de \un}{\de x_\alpha}\right)^2\, dx-\sum_{l=1}^{k-1} \frac{\big(s_n^l\big)^2}{\mn^l} \intO |\na  \vn^l|^2\, dx\\
&=\int_{\ball{R}{x_{i,n}}}\ln e^{u_n}\left(\frac{\de \un}{\de x_\alpha}\right)^2\, dx+o(1)-\sum_{l=m+1}^{k-1} \frac{\big(s_n^l\big)^2}{\mn^l} \intO |\na  \vn^l|^2\, dx.
\end{align*}
Therefore, similarly to the case $k=m+1$, it follows that
\begin{align*}
&\intO |\na \vn|^2\, dx=\intO \left|\na \left(\xi_n \frac{\de \un}{\de x_\alpha}\right)\right|^2\, dx-\sum_{l=1}^{k-1} \big(s_n^l\big)^2 \intO |\na  \vn^l|^2\, dx\\
&=\intO \left|\na \left(\xi_n \frac{\de \un}{\de x_\alpha}\right)\right|^2\, dx-\sum_{l=1}^{m} \big(s_n^l\big)^2 \intO |\na  \vn^l|^2\, dx-\sum_{l=m+1}^{k-1} \big(s_n^l\big)^2 \intO |\na  \vn^l|^2\, dx\\
&=\int_{\ball{R}{x_{i,n}}}\ln e^{u_n}\left(\frac{\de \un}{\de x_\alpha}\right)^2\, dx+O(1)-\sum_{l=m+1}^{k-1} \big(s_n^l\big)^2 \intO |\na  \vn^l|^2\, dx\\
&=\intO\ln e^{\un}\vn^2 dx+O(1)+\sum_{l=m+1}^{k-1} \frac{1-\mn^l}{\mn^l}\big(s_n^l\big)^2 \intO |\na  \vn^l|^2\, dx{\rosso{\text{ by }\eqref{***}}}.
\end{align*}
{\rosso{Then}} we get, {\rosso{from $iii)$ of  Proposition \ref{newp3.15}, \eqref{new3.15} and \eqref{claim3} }}
\begin{align*}
\sum_{l=m+1}^{k-1}& \frac{1-\mn^l}{\mn^l}\big(s_n^l\big)^2 \intO |\na  \vn^l|^2\, dx\\
&=O(\lambda_n)\cdot\left(\frac{a_{i,\alpha}^l}{d_i\nor\a^l\nor_{\R^{2m}}^2}+o(1)\right)^2\frac{1}{\ln}\cdot O(1)=O(1).
\end{align*}
%from the case 3 of Proposition \ref{newp3.15}, \eqref{new3.15}, and \eqref{claim3}.
Consequently it holds that
\begin{equation}\label{ciao}
\intO |\na \vn|^2\, dx=\intO\ln e^{\un}\vn^2 dx+O(1).
\end{equation}

On the other hand, 
%{\rosso{by \eqref{***}}}
 by \eqref{claim3},
 we get
\begin{align}
\intO\ln e^{\un}\vn^2 dx&=\int_{\ball{R}{x_{i,n}}}\ln e^{u_n}\left(\frac{\de \un}{\de x_\alpha}\right)^2\, dx+o(1)-\sum_{l=m+1}^{k-1} \frac{\big(s_n^l\big)^2}{\mn^l} \intO |\na  \vn^l|^2\, dx\nonumber\\
%&=\frac{4\pi}{3 d_i^2\ln}\left(1+o(1)\right)+o(1)-\sum_{l=m+1}^{k-1} \frac{\big(s_n^l\big)^2}{\mn^l} \intO |\na  \vn^l|^2\, dx\\
&=\frac{4\pi}{3 d_i^2\ln}\left(1+o(1)\right)-\sum_{l=m+1}^{k-1}\left(\frac{a^l_{i,\alpha}}{d_i\nor\a^l\nor_{\R^{2m}}^2}+o(1)\right)^2\frac{1}{\ln}\nonumber\\
&\qquad\qquad\times\left\{\frac{4\pi}{3}\nor\a^l\nor_{\R^{2m}}^2+o\left(1\right)\right\}+o(1)\nonumber\\
&=\frac{4\pi}{3d_i^2\ln}\left\{\left(1-\sum_{l=m+1}^{k-1}\frac{(a^l_{i,\alpha})^2}{\nor\a^l\nor_{\R^{2m}}^2}\right)+o(1)\right\}+o(1).\label{ciao2}
\end{align}

Let $\{e_{1,1},e_{1,2},\dots,e_{m,1},e_{m,2}\}$ be the canonical basis of $\R^{2m}$ and choose $(i,\alpha)\in \{1,\dots,m\}\times\{1,2\}$ satisfying
\begin{equation}\label{new4.4}
e_{i,\alpha}\notin\A(k-1):=\mathrm{span}\{\bm{a}^{m+1},\dots,\bm{a}^{k-1}\}.
\end{equation}
Since $\A(k-1)$ is a $(k-m-1)$-dimensional subspace of $\R^{2m}$ (see Proposition \ref{newp3.15}) and {\rosso{observing that $k\leq 3m$ implies that $2m-(k-m-1)\geq 1$, we have that there exists a pair  $(i,\alpha)$ satisfying \eqref{new4.4}}}.
Then \eqref{new4.4} implies that $\mathrm{proj}_{\A(k-1)^{\perp}}e_{i,\alpha}\neq 0$, and therefore
\begin{align*}
|\mathrm{proj}_{\A(k-1)^{\perp}}&e_{i,\alpha}|^2=|e_{i,\alpha}|^2-|\mathrm{proj}_{\A(k-1)}e_{i,\alpha}|^2\\
&=1-\left\{\sqrt{\sum_{l={\rosso{m+1}}}^{k-1}\left(e_{i,\alpha}\cdot \frac{\a^l}{\nor \a^l\nor_{\R^{2m}}}\right)^2}\right\}^2=1-\sum_{l={\rosso{m+1}}}^{k-1}\frac{(a_{i,\alpha}^l)^2}{\nor \a^l\nor _{\R^{2m}}^2}>0.
\end{align*}

Finally, it follows that, by \eqref{ciao} and \eqref{ciao2}
\begin{align*}
0\leq \mn^k&\leq\frac{\intO |\na v_n|^2 \, dx}{\ln \intO e^{\un} \vn^2 \, dx}=\frac{\ln \intO e^{\un} \vn^2 \, dx+O(1)}{\ln \intO e^{\un} \vn^2 \, dx}\\
&=1+\frac{O(1)}{\frac{4\pi}{3d_i^2\ln}\left\{\left(1-\sum_{l=m+1}^{k-1}\frac{(a^l_{i,\alpha})^2}{\nor\a^l\nor_{\R^{2m}}^2}\right)+o(1)\right\}+o(1)}\\
&=1+O(\lambda_n)
\end{align*}
and the proof is complete.
\end{proof}

%%%%%%%%%%%%%%%%%%%%%%%%%%%%%%%
%%%%%%%%%%%%%%%%%%%%%%%%%%%%%%%
\sezione{Proof of Theorem \ref{t1} and Theorem \ref{t1a}}\label{se-t}
%\subsection{$k$-th eigenvalue ($k=2,\dots,m$).}
\begin{proof}[Proof of Theorem \ref{t1}]
Propositions \ref{newp4.5} and \ref{newp3.16} imply
\begin{equation}\label{new3.15-1}
\mn^k=1-48\pi\til{\eta}^{k}\ln+o(\ln)
\end{equation}
for $m+1\leq k\leq 3m$, where $\til{\eta}^k$ is an eigenvalue of 
%${\mathrm{Hess}}(DH^mD)$
$D{\mathrm{Hess}}(H^m)D$.
 Then, {\rosso{from Proposition \ref{newp3.12}}}, it follows that $\b^k=\bm{0}$ for $m+1\leq k\leq 3m$.
Let $\mn^k\tend 1$ with $k\geq 3m+1$.  Then we have $\b^k\not=0$ because {\rosso{ $iv)$ of}} Proposition \ref{newp3.15} {\rosso{implies}} that $(\a^k,\b^k)\in\R^{3m}$ is orthogonal to the $2m$-dimensional subspace of $\R^{3m}$ spanned by $(\a^{m+1},\bm{0}),\cdots, (\a^{3m},\bm{0})$.  Hence we obtain
$$
\mn^k=1-\frac{3}{2}\frac{1}{\log\ln}+o\left(\frac{1}{\log\ln}\right)\;>1\quad\text{for $n\gg 1$.}
$$
by Proposition \ref{newp3.12}.  Thus, the calculation of the (augmented) Morse index of $u_n$ is reduced to {\rosso{study}} $\mu_m^k$ for $m+1\leq k\leq 3m$ because $\mu_\infty^k=0$ for $1\leq k\leq m$.

Since
$$
\mn^{m+1}\leq \cdots\leq\mn^{3m}
$$
we have
$$
\til{\eta}^{m+1}\geq\cdots\geq\til{\eta}^{3m}.
$$
Therefore, $\til{\eta}^{k}$ is the $\{2m-(k-m)+1\}$-th eigenvalue $\eta^{2m-(k-m)+1}$ of 
%$\mathrm{Hess}(DH^mD)$. 
$D(\mathrm{Hess}H^m)D$.

As we noticed in Section \ref{s1}, the {\rosso{sign}} of the $l$-th eigenvalue of
%\[{\mathrm{Hess}}\,(DH^m(\ka_1,\cdots,\ka_m)D) \]
\[D\{{\mathrm{Hess}}\,H^m(\ka_1,\cdots,\ka_m)\}D \]
and that of $\mathrm{Hess}\,H^m(\ka_1,\cdots,\ka_m)$ {\rosso{is}} the same (see Lemma \ref{sign_eigenvalue}). {\rosso{It is easy to see}} that the $\{2m-(k-m)+1\}$-th eigenvalue of 
%$\mathrm{Hess}(DH^mD)$
$D(\mathrm{Hess}H^m)D$
 is nothing but the $(k-m)$-th eigenvalue of 
%$\mathrm{Hess}\,\{(D(-H^m(\ka_1,\cdots,\ka_m))D)\}$
 $D\{\mathrm{Hess}\,(-H^m(\ka_1,\cdots,\ka_m)\}D$.
 {\rosso{Finally}}, it follows that
\begin{align*}
&\mathrm{ind}_M\{-H^m(\ka_1,\cdots,\ka_m)\}=\#\{k\in\N\,;\, \eta^k>0\},\\
&\mathrm{ind}^\ast_M\{-H^m(\ka_1,\cdots,\ka_m)\}=\#\{k\in\N\,;\, \eta^k\geq 0\}.
\end{align*}
Theorem \ref{t1} now follows from \eqref{new3.15-1}.
\end{proof}
\begin{proof}[Proof of Theorem \ref{t1a}]
From Proposition \ref{newp3.10} and Proposition \ref{p3.6} we get \eqref{10}. \eqref{10_1} follows by Proposition \ref{newp4.5} and Proposition \ref{newp3.16}. Finally \eqref{10_2} follows by Proposition \ref{newp3.12} and Remark \ref{newr3.13}.
\end{proof}
%%%%%%%%%%%%%%%%%%%%%%%%%%%%%%%
%%%%%%%%%%%%%%%%%%%%%%%%%%%%%%%%
%%%%%%%%%%%%%%%%%%%%%%%%%%%%%%%%
\bigskip
\appendix\section{Computation of integrals}

\begin{lemma}\label{l6.2}Let $U$ be as defined in \eqref{2.6}, $U_\alpha=\frac{\de U}{\de x_\alpha}$, and $\clo{U}=x\cdot\nabla U+2$. Then it holds that
\begin{align*}
&\int_{\R^2} e^U\, dx =8\pi,\quad
\int_{\R^2} e^U U\, dx =-16\pi,\quad\int_{\R^2} e^U U_\alpha\, dx=\int_{\R^2} e^U \clo{U}\, dx=0,\\
&\int_{\R^2} e^U U^2\, dx =64\pi,\quad\int_{\R^2} e^U UU_\alpha\, dx=0,\quad\int_{\R^2} e^U U\clo{U}\, dx=16\pi,\\
&\int_{\R^2} e^U U_\alpha U_\beta\, dx=\frac{4\pi}{3}\delta_{\alpha\beta},\quad\int_{\R^2} e^U U_\alpha \clo{U}\, dx=0,\quad\int_{\R^2} e^U \clo{U}^2\, dx=\frac{32\pi}{3}.
\end{align*}
\end{lemma}
The proof of the above formulae is elementary. %{\rosso{We show formula (\ref{eqn:22}).}}

\begin{proposition}
It holds that
\begin{equation}\label{eqn:22}
\int_{\R^2}\frac{1}{|\til{x}-\til{y}|}\frac{1}{\left(1+\frac{|\til{y}|^2}{8}\right)^2}d\til{y}\leq\frac{C}{1+|\til{x}|}, \quad \forall \til{x}\in\R^2.
\end{equation}
 \label{prop:a3}
\end{proposition}
\begin{proof}
We divide the proof into two cases.

\bigskip

\underline{Case 1}: For $|\til{x}|\leq 1$, we have
\begin{align*}
&\int_{\R^2}\frac{1}{|\til{x}-\til{y}|}\frac{1}{\left(1+\frac{|\til{y}|^2}{8}\right)^2}d\til{y}=\int_{\R^2\backslash\ball{1}{\til{x}}}+\int_{\ball{1}{\til{x}}}\\
&\leq\int_{\R^2}\frac{d\til{y}}{\left(1+\frac{|\til{y}|^2}{8}\right)^2}+\int_{\ball{1}{\til{x}}}\frac{d\til{y}}{|\til{x}-\til{y}|}=C\leq\frac{2C}{1+|\til{x}|}.
\end{align*}

\underline{Case 2}: For $|\til{x}|\geq 1$, we {\rosso{have}}
%\begin{align*}
%&\int_{\R^2}\frac{1}{|\til{x}-\til{y}|}\frac{1}{\left(1+\frac{|\til{y}|^2}{8}\right)^2}d\til{y}\leq\frac{C}{|\til{x}|}\left(\leq\frac{2C}{1+|\til{x}|}\right).
%\end{align*}
%To this end, we divide the integration into two parts:
\begin{align*}
&\int_{\R^2}\frac{1}{|\til{x}-\til{y}|}\frac{1}{\left(1+\frac{|\til{y}|^2}{8}\right)^2}d\til{y}=\int_{\R^2\backslash\ball{\frac{|\til{x}|}{2}}{\til{x}}}+\int_{\ball{\frac{|\til{x}|}{2}}{\til{x}}}\\
&\leq\int_{\R^2\backslash\ball{\frac{|\til{x}|}{2}}{\til{x}}}\frac{1}{\,\frac{|\til{x}|}{2}\,}\cdot\frac{1}{\left(1+\frac{|\til{y}|^2}{8}\right)^2}d\til{y}+\int_{\ball{\frac{|\til{x}|}{2}}{\til{x}}}\frac{1}{|\til{x}-\til{y}|}\cdot\frac{1}{\left\{1+\frac{1}{8}\left(\frac{|\til{x}|}{2}\right)^2\right\}^2}d\til{y}\\
&{\rosso{\leq \frac{16\pi}{|\til{x}|}}}+\frac{1}{\left(1+\frac{|\til{x}|^2}{32}\right)^2}\int_{\ball{\frac{|\til{x}|}{2}}{0}}\frac{d\til{y}}{|\til{y}|}\leq\frac{{\rosso{16\pi}}}{|\til{x}|}+\frac{{\rosso{1024\pi}}}{|\til{x}|^3}\leq \frac{{\rosso{1040\pi}}}{|\til{x}|},
\end{align*}
{\rosso{since}} $|\til{x}|\geq 1$.
\end{proof}
{\rosso{This estimate}} is a refinement of the one used in \cite[Lemma A.1]{GG09}.
\begin{lemma}\label{newl3.11}
There exists a constant $C$ independent of $n$ and $j$ satisfying
$$
\left|\frac{\de \til{u}_{j,n}}{\de \til{x}_\alpha}\right|\leq\frac{C}{1+|\til{x}|}\quad\text{in $\ball{\frac{R}{\d_{j,n}}}{0}$}.
$$
\end{lemma}
\begin{proof}
First, we have
\begin{align*}
&\frac{\de u_n(x)}{\de x_\alpha}=\int_\Omega G_{x_\alpha}(x,y)\ln e^{u_n}dy\\
&=\sum_{i=1}^m\int_{\ball{R}{x_{i,n}}}G_{x_\alpha}(x,y)\ln e^{u_n}dy+\int_{\Omega\backslash\cup_{i=1}^m\ball{R}{x_{i,n}}}G_{x_\alpha}(x,y)\ln e^{u_n}dy.
\end{align*}
Here it holds that
$$
\left|\int_{\Omega\backslash\cup_{i=1}^m\ball{R}{x_{i,n}}}G_{x_\alpha}(x,y)\ln e^{u_n}dy\right|\leq O(\ln)\int_\Omega\left|G_{x_\alpha}(x,y)\right|dy=O(\ln)
$$
and also, for $i\not=j$, that
$$
\int_{\ball{R}{x_{i,n}}}G_{x_\alpha}(x,y)\ln e^{u_n}dy=O(1)\quad\text{in $\ball{R}{x_{j,n}}$}.
$$
If $i=j$, on the other hand, we have
\begin{align*}
&\int_{\ball{R}{x_{j,n}}}G_{x_\alpha}(x,y)\ln e^{u_n}dy\\
&=-\frac{1}{2\pi}\int_{\ball{R}{x_{j,n}}}\frac{(x-y)_\alpha}{|x-y|^2}\ln e^{u_n}dy+\int_{\ball{R}{x_{j,n}}}K_{x_\alpha}(x,y)\ln e^{u_n}dy
\end{align*}
with
$$
\int_{\ball{R}{x_{j,n}}}K_{x_\alpha}(x,y)\ln e^{u_n}dy=8\pi K_{x_\alpha}(x,\ka_j)+o(1)=O(1)
$$
for $x\in\ball{R}{x_{j,n}}$.  Therefore, we have
\begin{align*}
\frac{\de \til{u}_{j,n}}{\de \til{x}_\alpha}&=\d_{j,n}\frac{\de u_n}{\de x_\alpha}(\d_{j,n}\til x+x_{j,n})\\
&=\d_{j,n}\left\{-\frac{1}{2\pi}\int_{\ball{\frac{R}{\d_{j,n}}}{0}}\frac{\d_{j,n}(\til{x}-\til{y})_\alpha}{\d_{j,n}^2|\til{x}-\til{y}|^2} e^{\til{u}_{j,n}}d\til{y}+O(1)\right\}
\end{align*}
and consequently,
$$
\left|\frac{\de \til{u}_{j,n}}{\de \til{x}_\alpha}\right|\leq\frac{1}{2\pi}\int_{\ball{\frac{R}{\d_{j,n}}}{0}}\frac{1}{|\til{x}-\til{y}|}e^{\til{u}_{j,n}}d\til{y}+O(\d_{j,n}).
$$

Here we may assume
$$
O(\delta_{j,n})\leq \frac{C}{1+|\til{x}|}
$$
with some $C>0$ by $\til{x}\in\ball{\frac{R}{\delta_{j,n}}}{0}$.  We have also
$$
\int_{\ball{\frac{R}{\d_{j,n}}}{0}}\frac{1}{|\til{x}-\til{y}|}e^{\til{u}_{j,n}}d\til{y}\leq C\int_{\R^2}\frac{1}{|\til{x}-\til{y}|}\frac{1}{\left(1+\frac{|\til{y}|^2}{8}\right)^2}d\til{y}
$$
with $C>0$ by \eqref{2.6a}. Therefore, the lemma follows from \eqref{eqn:22}.
\end{proof}

\section{A linear algebra lemma}
\begin{lemma}\label{sign_eigenvalue}
For every real symmetric $N\times N$ matrix $H$ and a real diagonal matrix $D=\mathrm{diag}[d_1,\cdots,d_N]$ with $d_j\not=0$ for every $j=1,\cdots, N$, let {\rosso{us denote by $\Lambda^k$ and $\til{\Lambda}^k$}} the $k$-th eigenvalue of $H$ and $DHD$. Then the {\rosso{sign}} of $\Lambda^k$ and $\til{\Lambda}^k$ {\rosso{is}} the same for every $k=1,\cdots,N$.
\end{lemma}
\begin{proof}
From the mini-max principle it follows that
$$
\til{\Lambda}^k= \min_{S_k}\max_{v\in S_k\backslash\{\bm{0}\}}
\frac{DHD[v]}{\nor v \nor_{\R^{N}}^2},
$$
where $S_k$ is a $k$-dimensional subspace of $\R^N$, $\min_{S_k}$ indicates the minimum among all of the $k$-dimensional subspaces, and $DHD[v]$ is the quadratic form $^tvDHDv$.

Since $D$ is also a symmetric matrix, it holds that
$$
DHD[v]=^tvDHDv=^tv\,^tDHDv=^t(Dv)H(Dv)=H[Dv]
$$
and
$$
\max_{v\in S_k\backslash\{\bm{0}\}}\frac{DHD[v]}{\nor v \nor_{\R^{N}}^2}=\max_{v\in S_k\backslash\{\bm{0}\}}\frac{H[Dv]}{\nor v \nor_{\R^{N}}^2}=\max_{v\in DS_k\backslash\{\bm{0}\}}\frac{H[v]}{\nor D^{-1}v \nor_{\R^{N}}^2}.
$$
Here we have
$$
\frac{\nor v \nor_{\R^{N}}^2}{\max_{1\leq j\leq N} d_j^2}\leq \nor D^{-1}v \nor_{\R^{N}}^2\leq \frac{\nor v \nor_{\R^{N}}^2}{\min_{1\leq j\leq N} d_j^2}.
$$
Therefore, it holds that
$$
\min_{1\leq j\leq N} d_j^2\max_{v\in DS_k\backslash\{\bm{0}\}}\frac{H[v]}{\nor v \nor_{\R^{N}}^2}\leq\max_{v\in DS_k\backslash\{\bm{0}\}}\frac{H[v]}{\nor D^{-1}v \nor_{\R^{N}}^2}\leq\max_{1\leq j\leq N} d_j^2\max_{v\in DS_k\backslash\{\bm{0}\}}\frac{H[v]}{\nor v \nor_{\R^{N}}^2},
$$
provided that $H[v]$ take a non-negative value, that is, the $k$-th eigenvalue $\Lambda^k\geq 0$. Consequently, it holds that
\begin{eqnarray*}
\min_{1\leq j\leq N} d_j^2\min_{S_k}\max_{v\in S_k\backslash\{\bm{0}\}}\frac{H[v]}{\nor v \nor_{\R^{N}}^2} & \leq & \min_{S_k}\max_{v\in DS_k\backslash\{\bm{0}\}}\frac{H[v]}{\nor D^{-1}v \nor_{\R^{N}}^2} \\
& \leq & \max_{1\leq j\leq N} d_j^2\min_{S_k}\max_{v\in S_k\backslash\{\bm{0}\}}\frac{H[v]}{\nor v \nor_{\R^{N}}^2},
\end{eqnarray*}
which implies
$$
\Lambda^k\min_{1\leq j\leq N} d_j^2\leq\til{\Lambda}^k\leq\Lambda^k\max_{1\leq j\leq N} d_j^2,
$$
when $\Lambda^k\geq 0$. Similarly, we obtain
$$
\Lambda^k\min_{1\leq j\leq N} d_j^2\geq\til{\Lambda}^k\geq\Lambda^k\max_{1\leq j\leq N} d_j^2
$$
if $\Lambda^k<0$.
\end{proof}

\bigskip
\footnotesize
\noindent\textit{Acknowledgments.}
The first two authors are supported by M.I.U.R., project
``Variational methods and nonlinear differential equations".
The third and the fourth authors are supported by Grant-in-Aid for Scientific Research (C) (No.22540231) and (B) (No. 20340034), Japan Society for the Promotion of Science.

\end{document}